\documentclass[11pt,reqno]{amsart}
\usepackage[utf8]{inputenc}

\usepackage[margin=1in]{geometry}
\usepackage{amsmath}
\usepackage{amsfonts}
\usepackage{amssymb}
\usepackage{amsthm}
\usepackage{enumitem}
\usepackage{mathrsfs}
\usepackage{mathtools}
\usepackage{tikz-cd}
\usepackage{enumitem}

\usepackage[colorlinks=true,hyperindex, linkcolor=magenta, pagebackref=false, citecolor=cyan,pdfpagelabels]{hyperref}

\newtheorem{theorem}{Theorem}[section]
\newtheorem{prop}[theorem]{Proposition}
\newtheorem{lemma}[theorem]{Lemma}
\newtheorem{cor}[theorem]{Corollary}

\theoremstyle{definition}
\newtheorem{definition}[theorem]{Definition}
\newtheorem{example}[theorem]{Example}
\newtheorem{remark}[theorem]{Remark}

\newtheorem{question}[theorem]{Question}

\renewcommand{\rm}{\mathrm}

\newcommand{\Spec}{\mathrm{Spec} \,}

\newcommand{\eps}{\epsilon}

\newcommand{\vp}{\varphi}

\newcommand{\bA}{{\mathbf A}}

\newcommand{\bF}{{\mathbf F}}
\newcommand{\bG}{{\mathbf G}}

\newcommand{\bP}{{\mathbf P}}
\newcommand{\bQ}{{\mathbf Q}}

\newcommand{\bZ}{{\mathbf Z}}

\newcommand{\bt}{{\mathbf t}}

\newcommand{\cG}{\mathcal{G}}

\newcommand{\cO}{\mathcal{O}}

\newcommand{\cT}{\mathcal{T}}

\newcommand{\sA}{{\mathscr A}}

\newcommand{\sC}{{\mathscr C}}
\newcommand{\sD}{{\mathscr D}}

\newcommand{\sL}{{\mathscr L}}

\newcommand{\sN}{{\mathscr N}}
\newcommand{\sO}{{\mathscr O}}

\newcommand{\sU}{{\mathscr U}}

\newcommand{\sX}{{\mathscr X}}
\newcommand{\sY}{{\mathscr Y}}

\newcommand{\fb}{\mathfrak{b}}

\newcommand{\fg}{\mathfrak{g}}

\newcommand{\fm}{\mathfrak{m}}
\newcommand{\fn}{\mathfrak{n}}

\newcommand{\fp}{\mathfrak{p}}
\newcommand{\fq}{\mathfrak{q}}
\newcommand{\fr}{\mathfrak{r}}

\newcommand{\fu}{\mathfrak{u}}

\newcommand{\fz}{\mathfrak{z}}

\newcommand{\bGa}{\bG_{\rm{a}}}
\newcommand{\bGm}{\bG_{\rm{m}}}

\DeclareMathOperator{\Ad}{Ad}

\DeclareMathOperator{\Aut}{Aut}
\DeclareMathOperator{\chara}{char}
\DeclareMathOperator{\Cl}{Cl}

\DeclareMathOperator{\diag}{diag}

\DeclareMathOperator{\GL}{GL}

\DeclareMathOperator{\Hom}{Hom}

\DeclareMathOperator{\id}{id}
\DeclareMathOperator{\im}{im}

\DeclareMathOperator{\Lie}{Lie}

\DeclareMathOperator{\Pic}{Pic}

\DeclareMathOperator{\PGL}{PGL}

\DeclareMathOperator{\SL}{SL}
\DeclareMathOperator{\SO}{SO}
\DeclareMathOperator{\Sp}{Sp}
\DeclareMathOperator{\Spin}{Spin}

\DeclareMathOperator{\Stab}{Stab}

\DeclareMathOperator{\Transp}{Transp}

\newcommand{\wtil}[1]{\wtil}

\newcommand{\ov}[1]{\overline{#1}}

\newcommand{\gp}[1]{#1\textrm{-}\rm{gp}}

\title{Central isogenies and conjugacy classes in reductive groups}
\author{Sean Cotner}

\begin{document}
\bibliographystyle{halpha-abbrv}

\begin{abstract}
    Steinberg described the group of components of the centralizer of a semisimple element of a connected semisimple algebraic group $G$ as a subgroup of the fundamental group of $G$. We show that this description can be generalized to explain the fact that centralizers of unipotent elements can fail to be reduced when the universal cover of $G$ is not \'etale. As applications, we compute generic multiplicities in the special fibers of moduli spaces of L-parameters and universal deformation rings, and we show there is no Springer isomorphism for $\PGL_p$ in characteristic $p$.
\end{abstract}

\maketitle

\section{Introduction}

Let $k$ be an algebraically closed field of characteristic $p \geq 0$, let $G$ be a simply connected semisimple $k$-group, and let $g \in G(k)$. Steinberg showed \cite[Theorem 8.1]{SteinbergEndomorphisms} that if $g$ is semisimple, then the centralizer $Z_G(g)$ is connected. More generally, if $\pi\colon G \to G'$ is a central isogeny, $g$ is semisimple, and $g' = \pi(g)$, then \cite[Theorem 9.1(a)]{SteinbergEndomorphisms} shows that $\pi_0(Z_{G'}(g'))$ is isomorphic to a subgroup of $(\ker \pi)(k)$. Indeed, $G'$ acts on $G$ by conjugation, so there is a natural map $Z_{G'}(g') \to \ker \pi$ given by $h \mapsto (hgh^{-1})g^{-1}$, and this induces a monic homomorphism 
\[
i_g\colon Z_{G'}(g')/\pi(Z_G(g)) \to \ker \pi.
\]
By Steinberg's connectedness result and the smoothness of $Z_{G'}(g')$ \cite[Corollary to Theorem 9.2]{Borel}, we have $Z_{G'}(g')/\pi(Z_G(g)) \cong \pi_0(Z_{G'}(g'))$. It is easy to check that if $\mu \subset (\ker \pi)(k)$ is a subgroup, then the following are equivalent:
\begin{enumerate}
    \item the image of $i_g$ contains $\mu$,
    \item the conjugacy class $C_g$ of $g$ is stable under multiplication by $\mu$.
\end{enumerate}

If $g$ is instead unipotent, then the map $i_g$ still makes sense, but it does not appear to be very interesting: indeed, the uniqueness of the Jordan decomposition shows that a unipotent orbit cannot be stable under multiplication by a nontrivial central element of $G(k)$. However, unlike in the case treated above, it is no longer true that $Z_{G'}(g')$ is smooth in general, even when $Z(G')$ is smooth. This may be seen most easily in the following example.

\begin{example}\label{example:pgl2-centralizer}
    Let $G' = \PGL_2$ over a field $k$ of characteristic $2$, let $u = \begin{pmatrix}
        1 &1 \\ 0 &1
    \end{pmatrix} \in \SL_2(k)$, and let $u'$ be the image of $u$ in $G'(k)$. Then \cite[Example 3.4]{Cotner-Springer} shows that $Z_{G'}(g')$ is not smooth. In fact, the map $i_u\colon Z_{G'}(u')/\pi(Z_{\SL_2}(u)) \to \mu_2$ is faithfully flat: this follows from monicity of $i_u$ and the fact that $\pi(Z_{\SL_2}(u)) \subset B'$ (by another direct calculation).
\end{example}

Thus we see that $i_g$ can be the trivial morphism on points, despite being scheme-theoretically nontrivial. For a subgroup \textit{scheme} $\mu \subset \ker \pi$, the equivalence of (1) and (2) still holds by the same argument, so we are led to the following precise question, where $Z(G)$ is the center of $G$.

\begin{question}\label{question:main}
    For given $g \in G(k)$, what is the scheme-theoretic stabilizer $\Stab_{Z(G)}(C_g)$?
\end{question}

To the author's knowledge, Question~\ref{question:main} has only been considered previously for semisimple $g$. We will answer Question~\ref{question:main} when $g$ is regular, which is already enough for some interesting applications. By the discussion above, it is most interesting to understand the stabilizer in $\Stab_{Z(G)^0}(C_g)$. If $H$ is an algebraic group over $k$, we denote by $\sD H$ the derived group of $H$.

\begin{theorem}[Theorem~\ref{theorem:stabilizer-of-Steinberg}]\label{theorem:intro-stability}
    If $g \in G(k)$ is regular with Jordan decomposition $g = tu$, then $\Stab_{Z(G)^0}(C_g) = \sD Z_G(t) \cap Z(G)^0$.
\end{theorem}

Note that in general, Question~\ref{question:main} is only interesting when $Z(G)$ is non-reduced, so it remains to understand the cases $G \cong \SL_n$ when $p \mid n$, $G \in \{\Sp_n, \Spin_n, \mathrm{E}_7\}$ when $p = 2$, or $G \cong \mathrm{E}_6$ when $p = 3$. Bootstrapping from Theorem~\ref{theorem:intro-stability}, we will answer Question~\ref{question:main} in Proposition~\ref{prop:stability-type-a} when $G = \SL_n$ and $g$ is unipotent. In the body of this paper, we will actually deal with a slightly more general question, involving twisted centralizers, where again we will answer the question only in a suitably 	``regular" case. However, the answer (Theorem~\ref{theorem:torsor-identification}) is rather complicated to state, so we leave it to the body of the paper.

The proof of Theorem~\ref{theorem:intro-stability} (and its generalization Theorem~\ref{theorem:torsor-identification}) involves reducing to the case that $g$ is unipotent, in which case we must show that the unipotent variety is $Z(G)^0$-stable. This is achieved as follows: let $\cO$ be a discrete valuation ring with residue field $k$ and characteristic $0$ fraction field $K$, and let $\cG$ is a split reductive $\cO$-group scheme with special fiber $G$. We lift $g$ to a section $x \in \cG(\cO)$ whose generic fiber has conjugacy class stabilized by translation by the full $p$-power-torsion subgroup of $Z(\cG)(K)$. Such a section $x$ is found using Steinberg's results cited above and \cite[Theorem 9.1(b)]{SteinbergEndomorphisms}. The regularity of $g$ is used to ensure that the special fiber of the closure of $C_{x_K}$ is equal to $C_g$, not just an infinitesimal thickening of $C_g$. (See Remark~\ref{remark:non-regular} and Question~\ref{question:finiteness}.)

We now describe two applications of Theorem~\ref{theorem:intro-stability}: first, we show that there is no Springer isomorphism for semisimple groups with non-\'etale universal cover; second, we give an application to multiplicities in moduli spaces of L-parameters, with an eye toward Shotton's $\ell \neq p$ version of the Breuil--M\'ezard conjecture.

\subsection{Unipotent schemes and Springer isomorphisms}

Continue to assume that $G$ is connected and semisimple, but drop the assumption that $G$ is simply connected. Let $G/\!/ G$ denote the GIT quotient for the conjugation action of $G$ on itself, and let $\chi = \chi_G\colon G \to G/\!/ G$ be the natural map. Similarly, let $\fg/\!/G$ be the GIT quotient for the conjugation of $G$ on its Lie algebra $\fg$, and let $\eta = \eta_G\colon \fg \to \fg/\!/G$ be the natural map. Recall that the \textit{unipotent variety} of $G$ is defined by $\sU_G^{\rm{var}} = \chi^{-1}(\chi(1))_{\rm{red}}$, and the \textit{nilpotent variety} of $G$ is defined by $\sN_G^{\rm{var}} = \eta^{-1}(\eta(0))_{\rm{red}}$.

When $p$ is ``good" for $G$ and does not divide $|\pi_1(G)|$, then there exists a $G$-equivariant isomorphism of varieties $\sU_G^{\rm{var}} \cong \sN_G^{\rm{var}}$ (see \cite{Springer-isomorphism}, \cite{Bardsley-Richardson}, or more recently \cite[Theorem 1.1]{Integral-Springer}). Springer also showed that there is no such isomorphism if $p$ is not good, and he showed in \cite[Remark 3.8]{Springer-isomorphism} (via a rather indirect argument) that there is no such isomorphism when $G = \PGL_2$ and $p = 2$; the same is shown in \cite[Corollary 7.0.3]{Sobaje-unip} using calculations of Taylor. The following example gives a more direct argument in this case, yielding a stronger result.

\begin{example}\label{example:pgl2}
Continuing Example~\ref{example:pgl2-centralizer}, let $G = \PGL_2$ over a field $k$ of characteristic $2$. An element of $\SL_2(k)$ is unipotent if and only if it has trace $0$ and determinant $1$, so we have
\[
\sU_{\SL_2}^{\rm{var}} = \left\{\begin{pmatrix} a & b \\ c & a \end{pmatrix}: a^2 - bc = 1 \right\}.
\]
By Theorem~\ref{theorem:intro-stability} (or an easy calculation), $\sU_{\SL_2}^{\rm{var}}$ is stable under left translation by the central $\mu_2 \subset \SL_2$, and it follows that the natural map $\sU_{\SL_2}^{\rm{var}} \to \sU_G^{\rm{var}}$ is a $\mu_2$-torsor. This shows
\[
\sU_G^{\rm{var}} = \sU_{\SL_2}^{\rm{var}}/\mu_2 \cong \left\{\begin{pmatrix} a & b \\ c & a \end{pmatrix}: a^2 - bc \neq 0 \right\}/\bGm \cong \bP^2 - \{a^2 = bc\}.
\]
A similar calculation shows that the map $\sN_{\SL_2}^{\rm{var}} \to \sN_G^{\rm{var}}$ is an $\alpha_2$-torsor and $\sN_G^{\rm{var}} \cong \bA^2$. These varieties are not isomorphic, since they have non-isomorphic Picard groups: $\Pic(\sU_G^{\rm{var}}) = \bZ/2\bZ$ and $\Pic(\sN_G^{\rm{var}}) = 0$.
\end{example}

As far as the author is aware, there are no examples in the literature beyond $\PGL_2$ for which $p$ is good and divides $|\pi_1(G)|$ and for which it is determined whether a Springer isomorphism exists for $G$. The following theorem deals with every such case in a rather strong way.

\begin{theorem}\label{theorem:intro-no-springer-iso}
    Suppose that $p$ divides $|\pi_1(G)|$.
    \begin{enumerate}
        \item (Corollary~\ref{cor:normal}) $\sU_G^{\rm{var}}$ is normal and lci.
        \item (Theorem~\ref{theorem:no-springer-iso}) There is no isomorphism of $k$-varieties $\sU_G^{\rm{var}} \not\cong \sN_G^{\rm{var}}$. In fact, if $\sN_G^{\rm{var}}$ is normal then $\Pic(\sU_G^{\rm{var}}) \not\cong \Pic(\sN_G^{\rm{var}})$.
        \item (Proposition~\ref{prop:no-dominant-map}) If $p$ is good for $G$, then there is no dominant $G$-equivariant $k$-morphism $\sU_G^{\rm{var}} \to \sN_G^{\rm{var}}$ or $\sN_G^{\rm{var}} \to \sU_G^{\rm{var}}$.
    \end{enumerate}
\end{theorem}

Theorem~\ref{theorem:intro-no-springer-iso}(1) was known when $p$ does not divide $|\pi_1(G)|$ for a long time; see the proof of \cite[Lemma 4.4]{Integral-Springer} and the references in \cite[Lemma 4.1]{Integral-Springer}. The key in general is to show that, if $\pi\colon \widetilde{G} \to G$ is the universal cover, then the induced map $\sU_{\widetilde{G}}^{\rm{var}} \to \sU_G^{\rm{var}}$ is a torsor for a finite $k$-group scheme (which does not even require Theorem~\ref{theorem:intro-stability}; see Lemma~\ref{lemma:stability-equivalences}), after which the result follows from descent. The reason for the strange statement of (2) is that I do not know whether $\sN_G^{\rm{var}}$ is always normal in bad characteristic; in any case, if it is not normal, then it follows from (1) that $\sU_G^{\rm{var}} \not\cong \sN_G^{\rm{var}}$. We do not calculate $\Pic(\sU_G^{\rm{var}})$ on the nose, and instead simply exhibit a ``large" subgroup of it; see Remark~\ref{remark:alpha2-torsor} and Remark~\ref{remark:guess-at-full-pic}. Strengthening (3), if $G \cong \PGL_p$ then we will show in Example~\ref{example:pglp-nonconstant} that there is no \textit{non-constant} $G$-equivariant $k$-morphism $\sU_G^{\rm{var}} \to \sN_G^{\rm{var}}$ or $\sN_G^{\rm{var}} \to \sU_G^{\rm{var}}$.

There are a number of other variations on this theme given in Section~\ref{section:counterexamples} when $p$ divides $|\pi_1(G)|$: for instance, in Theorem~\ref{theorem:regular-unipotent-centralizer} we show that if $u \in G(k)$ is regular unipotent, then $Z_G(u)$ is not commutative, and in Corollary~\ref{cor:unip-small-char} we show that the Springer resolution is not birational and $\chi^{-1}(\chi(1))$ is generically non-reduced. We prove some analogous results for nilpotents in Section~\ref{ss:nilp}. Using these results, we verify that various hypotheses in the results of \cite{Integral-Springer} are optimal.

\subsection{Moduli spaces of L-parameters}

Next, we move on to the moduli space of L-parameters. Let $F$ be a local field of (finite) residue cardinality $q$, and let $W_F$ be the Weil group of $F$. Let $A$ be a complete discrete valuation ring of mixed characteristic $(0, \ell)$, where $\ell \nmid q$, and let $G$ be a split reductive group scheme over $A$. Suppose that $W_F$ acts on $G$ through a finite quotient $W_0$, preserving a Borel pair $(B, T)$. For simplicity, in this introduction we will assume that $G$ is semisimple and the $W_F$-action on $G$ is unramified. Let $\pi\colon \widetilde{G} \to G$ be the universal cover, and let $\chi\colon G \to G/\!/G$ denote the GIT quotient map. Let $k$ be the residue field of $A$, which we will assume is perfect, let $\varpi$ be a uniformizer of $A$, and let $K$ be the fraction field of $A$. Let $\sX_G$ denote the moduli space of tame L-parameters $W_F \to G$ over $A$, as in \cite{DHKM}. (We will recall the definition in Section~\ref{ss:l-param}.)

If $X$ is an irreducible locally noetherian scheme, then we use the notation $\mu(X)$ to refer to the \textit{generic multiplicity} of $X$, a positive integer whose definition we recall in Definition~\ref{def:mult}. Shotton's $\ell \neq p$ version of the Breuil--M\'ezard conjecture \cite{Shotton-Breuil-Mezard}, \cite{Shotton-local-proof} involves studying $\mu(\sC_k)$, where $\sC$ is an open subscheme of the closure of an irreducible component of $(\sX_G)_K$ with irreducible special fiber. In \cite{Shotton-local-proof} and \cite{Shotton-irreducible}, Shotton shows that when $Z(G)$ is smooth and $|\pi_1(G)|$ is invertible in $A$, then for every ``$\Sigma$-regular" $\sC$ we have $\mu(\sC_k) = 1$. In \cite{Shotton-local-proof}, these are used in the case $G = \GL_n$ with trivial $W_F$-action to prove his conjecture (in a suitably 	``tame" case). The following theorem shows that there is no such equality if one weakens the hypotheses on $G$. 

\begin{theorem}\label{theorem:intro-l-param}
    Suppose that $K$ is large enough that every irreducible component of $(\sX_G)_K$ is geometrically irreducible and $(\ker \pi)(K) = (\ker \pi)(\ov{K})$. Let $C$ be an irreducible component of $(\sX_G)_K$, and let $D$ be a component of $(\ov{C} \cap \sX_G^{\Sigma\textrm{-}\rm{reg}})_k$. Let $g \in G(A)$ be a section with regular fibers such that the map $C \to G$, $(\Phi, \Sigma) \mapsto \Sigma$ factors through $\chi^{-1}(\chi(g))$. Let $\widetilde{g} \in \widetilde{G}(A)$ lift $g$, and let $\widetilde{g}_k = \widetilde{t}_k \widetilde{u}_k$ and $g_k = t_k u_k$ and $g_K = t_K u_K$ be the Jordan decompositions. Then $D$ is of generic multiplicity
        \[
        \mu(D) = \frac{|\pi_0(Z_{G_k}(t_k))| \cdot |(\ker \pi)_k^0 \cap \sD Z_{\widetilde{G}_k}(\widetilde{t}_k)| \cdot |Z_{\widetilde{G}_k}(\widetilde{g}_k)/((\ker \pi)_k \cdot Z_{\widetilde{G}_k}(\widetilde{g}_k)_{\rm{red}})|}{|\pi_0(Z_{G_K}(t_K))| \cdot |(Z_{\widetilde{G}_K}(\widetilde{g}_K)/((\ker \pi)_K \cdot Z_{\widetilde{G}_K}(\widetilde{g}_K)^0))[\ell^\infty]|},
        \]
        where $[\ell^\infty]$ refers to the subgroup of $\ell$-power torsion.
\end{theorem}

We will give a more general statement in Theorem~\ref{theorem:l-param-mult}, which includes the case that the $W_F$-action is ramified. (This is the main motivation for working with twisted stabilizers in the body of the paper.) In fact, Theorem~\ref{theorem:intro-l-param} combines the statements of Theorem~\ref{theorem:intro-stability} and Theorem~\ref{theorem:l-param-mult}; in the ramified case, the formula becomes even more complicated when combined with Theorem~\ref{theorem:torsor-identification}, so we do not state it explicitly.

\begin{example}\label{example:multiplicity-theorem-illustrated}
	We illustrate here Theorem~\ref{theorem:intro-l-param} in the cases of $G = \SL_2$ and $G = \PGL_2$, where $W_F$ acts trivially on $G$. First, assume $G = \SL_2$. In this case, $G$ is simply connected, so we get
	\[
	\mu(D) = \frac{|Z_{G_k}(g_k)/Z_{G_k}(g_k)_{\rm{red}}|}{|(Z_{G_K}(g_K)/(Z_{G_K}(g_K)^0))[\ell^\infty]|}.
	\]
	All centralizers in $\SL_2$ are smooth except in characteristic $2$, so if $\ell \neq 2$ then $\mu(D) = 1$. If $\ell = 2$, then $Z_{G_k}(g_k)$ is reduced unless $g_k$ is regular unipotent, and in that case we have $|Z_{G_k}(g_k)/Z_{G_k}(g_k)_{\rm{red}}| = 2$. If moreover $g_K$ is regular unipotent, then we have $\mu(D) = 1$, and in all other cases $\mu(D) = 2$.

    Now consider the case $G = \PGL_2$. Note that $\widetilde{G} = \SL_2$, and $\pi\colon \SL_2 \to G$ is the natural map. If $\ell \neq 2$, then the expression in Theorem~\ref{theorem:intro-l-param} simplifies to 
    \[
        \mu(D) = \frac{|\pi_0(Z_{G_k}(t_k))|}{|\pi_0(Z_{G_K}(t_K))|}.
    \]
    If $L$ is a field, then up to conjugacy the only element of $\SL_2(L)$ whose conjugacy class has nontrivial stabilizer in $\mu_2$ is $g_0 \coloneqq\begin{pmatrix} i &1 \\ &-i \end{pmatrix}$, where $i^2 = -1$. Thus $\mu(D) = 1$ unless $g_k$ is conjugate to $g_0$. If $g_k$ is conjugate to $g_0$, then $\mu(D) = 2$ unless $g_K$ is also conjugate to $g_0$, in which case $\mu(D) = 1$.
    
    If instead $\chara k = 2$, then the formula simplifies to
    \[
    \mu(D) = \frac{|\mu_2 \cap \sD Z_{\widetilde{G}_k}(\widetilde{t}_k)|}{|\pi_0(Z_{G_K}(t_K))|}.
    \]
    The term $|\mu_2 \cap \sD Z_{\widetilde{G}_k}(\widetilde{t}_k)|$ is $1$ unless $\widetilde{t}_k = 1$, i.e., $g_k$ is regular unipotent, and $|\pi_0(Z_{G_K}(t_K))| = 1$ unless $t_K$ is conjugate to $g_0$, so again $\mu(D) = 1$ unless $g_k$ is conjugate to $g_0$. If $g_k$ is conjugate to $g_0$, then again $\mu(D) = 2$ unless $g_K$ is also conjugate to $g_0$, in which case $\mu(D) = 1$.
\end{example}

The $\ell \neq p$ Breuil--M\'ezard conjecture has not been formulated beyond $\GL_n$, and we will not do so here, but we plan to return to it in the future. One may hope that, in spite of Theorem~\ref{theorem:intro-l-param}, perhaps at least the irreducible components of the Galois deformation rings of generic tame L-parameters with regular monodromy are formally smooth. If $Z(G)$ is smooth and $\pi$ is \'etale, then this follows from Shotton's results. However, we show in Example~\ref{example:universal-deformation-ring} that this is not the case for $G = \PGL_2$ in residue characteristic $2$, and in fact it is not true for $G$ of adjoint type if the universal cover is not \'etale. In fact, the arguments show that the minimally ramified component can fail to be formally smooth in this case.

We begin with some generalities on commutative algebra and descent for line bundles in Section~\ref{section:prelim}. In Section~\ref{section:twisted}, we study twisted centralizers in general, and we correct a result in the literature. The main results described above are then contained in Sections~\ref{section:stability}, \ref{section:counterexamples}, and \ref{section:moduli}.

\subsection{Notation and conventions}

If $A$ is an abelian group and $n$ is a positive integer, then $A[n^\infty]$ denotes the subgroup of $A$ consisting of those $a \in A$ such that $n^m a = 0$ for some $m \geq 0$. Similarly, if $G$ is a finite commutative group scheme over a field $k$, then $G[n^\infty]$ denotes the $k$-subgroup scheme of $n$-power torsion.

If $k$ is a field, then we use $\ov{k}$ to denote an algebraic closure of $k$.

If $X$ is a scheme, then $X_{\rm{red}}$ refers to the underlying reduced subscheme of $X$.

If $R$ is a ring and $X$ is an affine $R$-scheme, then $R[X]$ denotes the coordinate ring of $X$. If $R$ is a discrete valuation ring, then $X_\eta$ denotes the generic fiber of $X$ over $\Spec R$. If $R$ is a field and $X$ is $R$-finite, then $|X|$ denotes $\dim_R R[X]$.

\subsection{Acknowledgements}

I thank Ben Church, Spencer Dembner, Vaughan McDonald, and Jack Shotton for helpful conversations. I thank LoG(M) students Veer Agarwal, Linkun Ma, and Yan Yu for using Macaulay2 to help with a calculation in Remark~\ref{remark:twisted-steinberg-fiber}. This material is based upon work supported by the National Science Foundation under Award No.\ 2402231.

\section{Preliminaries}\label{section:prelim}

We begin with some technical results concerning generic multiplicities and descent for line bundles. We advise the reader to skip this section on a first reading, referring back to it as necessary in the sequel.

\subsection{Commutative algebra}

The main point of this section is to prove Lemma~\ref{lemma:multiplicity-behavior}, though we include a few other generalities. We use $\ell_A(M)$ to denote the length of an $A$-module $M$.

\begin{lemma}\label{lemma:length-behavior-2}
    If $(B, \fm) \to (C, \fn)$ is a morphism of artin local rings, then
    \begin{equation}\label{eqn:length-behavior-2-1}
    \ell_B(C) = \ell_{B/\fm}(C/\fn) \cdot \ell_C(C).
    \end{equation}
    If $C$ is flat over $B$, then
    \begin{equation}\label{eqn:length-behavior-2-2}
    \ell_C(C) = \ell_B(B) \cdot \ell_{C/\fm C}(C/\fm C).
    \end{equation}
\end{lemma}

\begin{proof}
    Let $0 = C_0 \subset C_1 \cdots \subset C_n = C$ be a filtration such that $C_i/C_{i-1} \cong C/\fn$ for $1 \leq i \leq n$, so $\ell_C(C) = n$. Since $\ell_B(C/\fn) = \ell_{B/\fm}(C/\fn)$, it follows that $\ell_B(C) = n \cdot \ell_{B/\fm}(C/\fn)$, proving (\ref{eqn:length-behavior-2-1}). For (\ref{eqn:length-behavior-2-2}), note that flat modules over artin local rings are free, so $C \cong B^m$ as $B$-modules for some $m$, and $\ell_B(C) = m\ell_B(B)$. Equation (\ref{eqn:length-behavior-2-1}) applied to $B/\fm \to C/\fm C$ shows
    \[
    m = \ell_{B/\fm}(C/\fm C) = \ell_{B/\fm}(C/\fn) \cdot \ell_{C/\fm C}(C/\fm C).
    \]
    Combining these equations with (\ref{eqn:length-behavior-2-1}) yields (\ref{eqn:length-behavior-2-2}).
\end{proof}

\begin{definition}\label{def:mult}
    Let $X$ be an irreducible locally noetherian scheme, and let $\eta$ be the generic point of $X$. We define the \textit{generic multiplicity} $\mu(X)$ of $X$ as $\ell_{\sO_{X, \eta}}(\sO_{X, \eta})$ (also known as the Hilbert--Samuel multiplicity of $\sO_{X,\eta}$).
\end{definition}

Note that, if $U$ is an open subscheme of $X$, then $\mu(U) = \mu(X)$. 

\begin{lemma}\label{lemma:mult-flat-cover}
    Let $f\colon X \to Y$ be a flat morphism of irreducible noetherian schemes. We have
    \[
    \mu(X) = \mu(Y) \cdot \mu(f^{-1}(Y_{\rm{red}})).
    \]
\end{lemma}

\begin{proof}
    Localizing around the generic points of $X$ and $Y$, we may assume $X = \Spec B$ and $Y = \Spec C$ for artin local rings $B$ and $C$. The result is then precisely Lemma~\ref{lemma:length-behavior-2}(\ref{eqn:length-behavior-2-2}).
\end{proof}

If $A$ is a DVR with uniformizer $\pi$ and $M$ is an $A$-module, then we write $\ov{M} = M/\pi M$.

\begin{lemma}\label{lemma:length-behavior-1}
    Let $A$ be a DVR with uniformizer $\pi$, let $B$ be a flat noetherian $A$-algebra with $\ov{B}$ artinian, and let $M$ be a finite $A$-flat $B$-module such that $M[1/\pi] \cong B[1/\pi]^n$. We have
    \[
    \ell_{\ov{B}}(\ov{M}) = n\ell_{\ov{B}}(\ov{B}).
    \]
\end{lemma}

\begin{proof}
    Let $N \subset M$ be a free $B$-submodule of rank $n$, which exists because $M$ is finite over $B$ and $M \subset M[1/\pi] \cong B[1/\pi]^n$. Since $B$ is noetherian and $\ov{B}$ is artinian, it follows that $M/N$ is of finite length over $B$, so from the exact sequence
    \[
    0 \to (M/N)[\pi] \to M/N \xrightarrow{\pi} M/N \to \ov{M/N} \to 0
    \]
    we get
    \[
    \ell_{\ov{B}}((M/N)[\pi]) = \ell_B((M/N)[\pi]) = \ell_B(\ov{M/N}) = \ell_{\ov{B}}(\ov{M/N}).
    \]
    Since $M$ is $A$-flat, we have another exact sequence
    \[
    0 \to (M/N)[\pi] \to \ov{N} \to \ov{M} \to \ov{M/N} \to 0
    \]
    and hence
    \[
    \ell_{\ov{B}}(\ov{M}) = \ell_{\ov{B}}(\ov{N}) = \ell_{\ov{B}}(\ov{B}^n) = n\ell_{\ov{B}}(\ov{B}),
    \]
    as desired.
\end{proof}

\begin{definition}\label{def:deg}
    Let $f\colon X \to Y$ be a quasi-finite morphism of integral locally noetherian schemes. If $\gamma$ is the generic point of $Y$, then the \textit{degree} $\deg(f)$ of $f$ is the rank of $f^{-1}(\gamma)$ over $k(\gamma)$.
\end{definition}

\begin{lemma}\label{lemma:multiplicity-behavior}
    Let $A$ be a DVR, and let $f\colon \sX \to \sY$ be a finite morphism of flat noetherian $A$-schemes such that $f_\eta$ is free of rank $n$. If $\ov{\sY}$ is irreducible and $\ov{\sX}_1, \dots, \ov{\sX}_m$ are the irreducible components of $\ov{\sX}$, then
    \[
    \mu(\ov{\sY}) = \frac{1}{n}\sum_{i=1}^m \deg(\ov{f}_i) \mu(\ov{\sX}_i),
    \]
    where $\ov{f}_i$ is the restricted map $(\ov{\sX}_i)_{\rm{red}} \to \ov{\sY}_{\rm{red}}$.
\end{lemma}

\begin{proof}
    Localizing around the generic point of $\ov{\sY}$, we can pass to the case that $\sY = \Spec B$ for a flat noetherian local $A$-algebra $B$ such that $\ov{B}$ is artin local. If $\sX = \Spec C$, then by assumption $C[1/\pi] \cong B[1/\pi]^n$, so by Lemma~\ref{lemma:length-behavior-1} it follows that
    \[
    n\cdot\mu(\ov{\sY}) = n\ell_{\ov{B}}(\ov{B}) = \ell_{\ov{B}}(\ov{C}) = \sum_{i=1}^m \ell_{\ov{B}}(\ov{C}_i),
    \]
    where $\ov{\sX}_i = \Spec \ov{C}_i$. Moreover, by Lemma~\ref{lemma:length-behavior-2},
    \[
    \ell_{\ov{B}}(\ov{C}_i) = \ell_{\ov{B}_{\rm{red}}}((\ov{C}_i)_{\rm{red}}) \cdot \ell_{\ov{C}_i}(\ov{C}_i) = \deg(\ov{f}_i) \cdot \mu(\ov{\sX}_i).
    \]
    Combining the displayed equations gives the desired result.
\end{proof}

\begin{remark}
    Lemma~\ref{lemma:multiplicity-behavior} is purely a numerical statement; the assumptions of the lemma do not seem to put any further restrictions on the structure of the local ring of the generic point of $\ov{\sY}$.
\end{remark}

The following lemma is only used in Example~\ref{example:universal-deformation-ring}.

\begin{lemma}\label{lemma:preimage-of-completion}
    Let $f\colon X \to Y$ be a finite morphism of locally noetherian schemes. Let $y \in Y$, and suppose that $f^{-1}(y)$ is topologically a single point $x$. Then the natural map $f^{-1}(\Spec \widehat{\cO}_{Y, y}) \to \Spec \widehat{\cO}_{X, x}$ is an isomorphism.
\end{lemma}

\begin{proof}
    Working locally, we may assume $X = \Spec B$ and $Y = \Spec A$, where $B$ is a finite $A$-algebra. Let $\fp \subset A$ and $\fq \subset B$ correspond to $y$ and $x$, respectively. Since $f$ is finite and $B_\fq/\fp B_\fq$ is local and finite over $A_\fp/\fp$, it follows that $\sqrt{\fp B} = \fq$ and thus the natural map $A_{\fp} \otimes_A B \to B_{\fq}$ is an isomorphism. It follows that $\widehat{A}_\fp \otimes_A B \to \widehat{B}_\fq$ is an isomorphism: indeed, we have 
    \[
    (\widehat{A}_\fp \otimes_A B)/\fp^N = A_\fp/\fp^N \otimes_A B = B_\fq/\fp^N B_\fq = \widehat{B}_\fq/\fp^N \widehat{B}_\fq
    \]
    for all $N \geq 1$, so upon passing to inverse limits we obtain
    \[
    \widehat{A}_\fp \otimes_A B = \varprojlim_N (\widehat{A}_\fp \otimes_A B)/\fp^N = \widehat{B}_\fq,
    \]
    where the first equality follows from the fact that $\widehat{A}_\fp \otimes_A B$ is a finite, hence $\fp$-adically complete, $\widehat{A}_\fp$-module.
\end{proof}

Recall that if $X$ is a scheme locally of finite type over a field $k$ and $x \in X(k)$, then the \textit{embedding dimension} of $X$ at $x$, which we will denote by $e_X(x)$, is the dimension of the Zariski cotangent space $\fm_x/\fm_x^2$. Equivalently, $e_X(x)$ is the least integer $d$ such that there exists a neighborhood $U$ of $x$ in $X$ and an unramified morphism $U \to \bA_k^d$.

\begin{lemma}\label{lemma:emb-dim-ineq}
    Let $A$ be a discrete valuation ring with fraction field $K$ and residue field $k$, and let $\sX$ be a finite type $A$-scheme. If $x \in \sX(A)$, then we have $e_{\sX_k}(x_k) \geq e_{\sX_K}(x_K)$.
\end{lemma}

\begin{proof}
    Let $d = e_{\sX_k}(x_k)$, and let $U$ be an affine open neighborhood of $x$ in $\sX$ such that there exists an unramified morphism $U_k \to \bA_k^d$. This map corresponds to a ring homomorphism $k[t_1, \dots, t_d] \to k[U_k]$, and we may lift this to a ring homomorphism $A[t_1, \dots, t_d] \to A[U]$, corresponding to an $A$-morphism $f\colon U \to \bA_A^d$. Since the unramified locus of $f$ is open, it follows that $f_K$ is unramified at $x_K$, so $e_{\sX_K}(x_K) \leq d$, as desired.
\end{proof}

\begin{lemma}
	Let $f\colon X \to Y$ be a finitely presented finite morphism of schemes such that $Y$ is reduced and the map $y \mapsto |f^{-1}(y)|$ is locally constant on $Y$. Then $f$ is flat.
\end{lemma}

\begin{proof}
	Working locally on $Y$ and spreading out, we may assume that $Y = \Spec A$ and $X = \Spec B$, where $B$ is finite $A$-algebra and $A$ is a noetherian local ring with maximal ideal $\fm$. Choose $b_1, \dots, b_n \in B$ whose images in $B/\fm B$ form an $A/\fm$-basis. By Nakayama's lemma, the map $A^n \to B$ given by $(a_1, \dots, a_n) \mapsto \sum_{i=1}^n a_ib_i$ is surjective. This gives rise to a short exact sequence
	\[
	0 \to K \to A^n \to B \to 0.
	\]
	If $\fp \subset A$ is a minimal prime ideal, then the residue field $k(\fp)$ is flat over $A$ by reducedness, so $K \otimes_A k(\fp) = 0$. Let $k = \prod_\fp k(\fp)$ be the total ring of fractions of $A$, which is the product of the residue fields of $A$ at minimal primes $\fp$; note that the product is finite since $A$ is noetherian. So
	\[
	K \otimes_A k = \prod_{\fp \text{ minimal}} K \otimes_A k(\fp) = 0. 
	\]
	Since $k$ is the localization of $A$ at the set of nonzerodivisors of $A$, it follows that every element of $K$ is killed by some nonzerodivisor of $A$. But $A^n$ is torsion-free, so this implies $K = 0$.
\end{proof}

\subsection{Descent for line bundles}

Recall that if $\pi\colon Y \to X$ is an fpqc morphism of schemes, then quasicoherent sheaves on $X$ can be described in terms of \textit{descent data} for quasicoherent sheaves on $Y$; this will be lightly recalled in the proof of the following lemma, but see \cite[Chapter 6]{BLR} for details. In particular, one can describe the kernel of $\pi^*\colon \Pic(X) \to \Pic(Y)$ in terms of descent data for the trivial line bundle on $Y$. In this section, we make such descent data explicit when $Y \to X$ is a torsor for a group scheme over a field.

\begin{lemma}\label{lemma:descent}
	Let $k$ be a field, let $H$ be a finite type $k$-group scheme, and let $\pi\colon Y \to X$ be an $H$-torsor of $k$-schemes such that $\bGm(Y) = k^\times$. Then there is a natural isomorphism between $\ker(\Pic(X) \to \Pic(Y))$ and the set of $\psi \in \bGm(H \times_k Y)$ satisfying
	\begin{align}\label{align:relation}
	\psi(g, hy)\psi(h, y) = \psi(gh, y)
	\end{align}
	for all local sections $g, h$ of $H$ and $y$ of $Y$. In particular, there is a natural injective homomorphism
	\[
	i\colon \Hom_{\gp{k}}(H, \bGm) \to \ker(\Pic(X) \to \Pic(Y)).
	\]
\end{lemma}

\begin{proof}
	Let $\pi_i: Y \times_X Y \to Y$ and $\pi_{ij}: Y \times_X Y \times_X Y \to Y \times_X Y$ be the natural projections. By fpqc descent for quasicoherent sheaves, a line bundle on $X$ (up to isomorphism) is equivalent to the data of a line bundle $\sL$ on $Y$ along with an isomorphism $\vp: \pi_1^*\sL \to \pi_2^*\sL$ satisfying the cocycle condition $\pi_{23}^*\vp \circ \pi_{12}^*\vp = \pi_{13}^*\vp$ (considered up to isomorphism of such data). Elements of $\ker(\Pic(X) \to \Pic(Y))$ therefore correspond to elements of $\bGm(Y \times_X Y)$ satisfying a cocycle condition, up to isomorphism. Using the isomorphism $H \times_{\Spec k} Y \to Y \times_X Y$ given by $(h, y) \mapsto (y, hy)$ and unraveling the cocycle condition, we find that $\vp$ corresponds to a morphism $\psi: H \times_{\Spec k} Y \to \bGm$ satisfying (\ref{align:relation}), up to isomorphism. Furthermore, $\psi$ corresponds to the trivial line bundle on $X$ if and only if $\psi$ is obtained as obtained as the pullback of an element $a \in \bGm(Y) = k^\times$. For such an $a$, (\ref{align:relation}) shows $a^2 = a$, so in fact $a = 1$, i.e., there are no nontrivial isomorphisms between two different descent data as above. The final claim follows from the fact that, given a $k$-homomorphism $\chi\colon H \to \bGm$, one can define $\psi$ as above by $\psi(h, y) = \chi(h)$.
\end{proof}

\begin{remark}\label{remark:alpha2-torsor}
If $Y$ is proper or $k$ is algebraically closed and $H$ is smooth, then the map $i$ of Lemma~\ref{lemma:descent} is an isomorphism, but under weaker hypotheses it is not.

For example, take $p = 2$, let $k$ be a field of characteristic $2$, let $Y_0$ be a smooth proper variety over $k$ on which $\mu_2$ acts freely (e.g., an ordinary elliptic curve with full rational $2$-torsion), and let $Y = Y_0 \times \Spec k[a, b]/(a^2 + b^2 - a)$. Note that $\mu_2$ acts on $\Spec k[a, b]/(a^2 + b^2 - a)$ by acting trivially on $a$ and through the unique nontrivial character $\chi$ on $b$. Thus the diagonal action of $\mu_2$ on $Y$ is free, and the map $Y \to X \coloneqq Y/\mu_2$ is a $\mu_2$-torsor. Moreover, we have $\Spec k[a, b]/(a^2 + b^2 - a) \cong \bA^1$, so properness of $Y_0$ implies $k[Y]^\times = k^\times$. Unraveling Lemma~\ref{lemma:descent} shows that an element of $\ker(\Pic(X) \to \Pic(Y))$ is equivalent to the data of $f \in k[Y]_1$, $g \in k[Y]_\chi$ such that $f^2 + g^2 = f$. In this situation, one can take $f$ (resp.\ $g$) to be the composition of $a$ (resp.\ $b$) with the projection to $\Spec k[a, b]/(a^2 + b^2 - a)$ to show that $\ker(\Pic(X) \to \Pic(Y)) \neq \Hom_{\gp{k}}(\mu_2, \bGm) = \bZ/2$.

As another example, if $p = 2$ then there is an $\alpha_2$-torsor
\[
Y \coloneqq \{(x, y): x^{10} + y^6 + x^2y^2 + x^4 + y = 0\} \to X \coloneqq \{(x, y): x^5 + y^6 + xy^2 + x^2 + y = 0\}
\]
between smooth affine curves inside $\bA^2$, sending $(x, y)$ to $(x^2, y)$. Using the Jacobian criterion, one checks that the closure $\overline{X}$ of $X$ inside $\bP^2$ is a smooth curve with precisely one point at infinity, so $\bGm(X) = k^\times$. Since $Y$ is smooth and connected and $Y \to X$ is finite, it follows that $\bGm(Y) = k^\times$. Unraveling the conditions on a descent datum for a line bundle, we find that line bundles on $X$ pulling back to the trivial bundle on $Y$ correspond to functions $f$ in $k[Y]$ such that $\frac{\partial{f}}{\partial{x}} = f^2$. Taking $f(x, y) = x^5 + y^3 + xy$ shows that $\Pic(X) \to \Pic(Y)$ is not injective.
\end{remark}

\section{Twisted centralizers}\label{section:twisted}

Throughout this section, we will fix a scheme $S$ (which, after Section~\ref{ss:flat}, will be the spectrum of a field) and a reductive $S$-group scheme $G$. Let $\sigma$ be an $S$-automorphism of $G$ which preserves a pinning $(B, T, \{X_\alpha\}_{\alpha \in \Delta})$.\footnote{This is not necessary and is only done for convenience; one can always reduce to this case by using the fact that every $S$-automorphism of $G$ can be written as the composition of an inner automorphism with a pinning-preserving one. Thus we will use the results of this section below even for non-pinning-preserving automorphisms.} Let $W$ denote the Weyl group of $(G, T)$. For $\alpha \in \Delta$, let $x_\alpha\colon \bGa \to U_\alpha$ be the isomorphism corresponding to $X_\alpha$. Use $\sD G$ to denote the derived group of $G$.

\subsection{Flatness}\label{ss:flat}

For $g \in G(S)$, let $Z_{G,\sigma}(g)$ denote the twisted centralizer scheme, so for each $S$-scheme $S'$ the set $Z_{G,\sigma}(g)(S')$ consists of $h \in G(S')$ such that $hg = g\sigma(h)$. In this section, we prove the flatness of twisted centralizers over the $\sigma$-regular locus. For the purposes of Section~\ref{ss:smooth}, an important consequence is that if $G$ is semisimple and simply connected, $S = \Spec k$ for a field $k$, and $b \in B(k)$ is $\sigma$-regular, then $Z_{B,\sigma}(b) \subset B$. This is claimed in \cite[Proposition 5.2.19(4)]{Xiao-Zhu}, but the proof there is not complete, as we now explain.

Write $B = TU$, where $U$ is the unipotent radical of $B$. First, \cite[Proposition 5.2.4(2) and (4)]{Xiao-Zhu} claim that if $u \in U(k)$ is $\sigma$-regular then $Z_{G,\sigma}(u) \subset B$. However, this is not correct: in fact, if $\sigma = 1$, then Theorem~\ref{theorem:regular-unipotent-centralizer} will show that $Z_G(u) \subset B$ if and only if $\chara k \nmid |\pi_1(\sD G)|$ (and Example~\ref{example:pgl2-centralizer} has already illustrated that $Z_G(u)$ can fail to lie in $B$). When $\sigma \neq 1$, there are other issues related to $Z(G)$, which the following example shows.

\begin{example}
    Let $\sigma$ be the automorphism of $G = \GL_2$ given by $\sigma(g) = \det(g)^{-1} g$, and suppose $\chara k = 2$. In this case, $(1-\sigma)Z \cap \sD G = \mu_2$. If $u = \begin{pmatrix} 1 &1 \\ 0 &1 \end{pmatrix}$, then we have
    \[
    Z_{G,\sigma}(u) \cong \{g \in \SL_2\colon gug^{-1} \in \mu_2 \cdot u\} = \pi^{-1}(Z_{\PGL_2}(\pi(u))),
    \]
    where $\pi\colon \SL_2 \to \PGL_2$ is the natural quotient. By Example~\ref{example:pgl2-centralizer}, the centralizer $Z_{\PGL_2}(\pi(u))$ does not lie in any Borel and in fact $Z_{\PGL_2}(\pi(u))$ is not smooth.
\end{example}

The main issue with the proof of \cite[Proposition 5.2.4(2)]{Xiao-Zhu} is that \cite[Lemma 5.2.5]{Xiao-Zhu} is partially incorrect: the claims about $\fu_x$ and $\fb_x$ are correct, but the final claim that $\fg_x = \fb_x$ is only true when $G$ is semisimple and simply connected, as the above examples show. Since \cite[Lemma 5.2.5]{Xiao-Zhu} is used in the proof of \cite[Proposition 5.2.19(4)]{Xiao-Zhu}, the proof of that result is also incomplete. However, the latter result is still true, as we will show. The other results of \cite{Xiao-Zhu} (which we use extensively below) are unaffected by these issues.

If $g \in G(k)$, let $c_g\colon G \to G$ be the map $c_g(h) = ghg^{-1}$. If $\tau = c_g \sigma$, then $Z_{G,\sigma}(g) = G^\tau$, so the study of twisted centralizers is equivalent to the study of the fixed points of automorphisms. The latter perspective is somewhat cleaner, but for the applications to L-parameters in this paper it is for the most part more convenient to work with twisted centralizers.

Recall the morphism $\chi_\sigma\colon G \to G/\!/_\sigma G$ to the twisted GIT quotient. By \cite[Theorem 3.10]{Cotner-conn-comps}, the formation of $G/\!/_\sigma G$ commutes with all base change, and the natural map $T_\sigma/W^\sigma \to G/\!/_\sigma G$ is an isomorphism. 
Let $G^{\sigma\textrm{-}\rm{reg}}$ denote the open subscheme of $G$ such that for every $S$-scheme $S'$, the set $G^{\sigma\textrm{-}\rm{reg}}(S')$ consists of those $g \in G(S')$ such that $Z_{G_{S'},\sigma}(g)$ is of the same (constant) relative dimension as $T^\sigma$.

\begin{lemma}\label{lemma:steinberg-map-flat}
    If $G$ is semisimple and simply connected, then $\chi_\sigma$ is flat and $\chi_\sigma|_{G^{\sigma\textrm{-}\rm{reg}}}$ is smooth.
\end{lemma}

\begin{proof}
    By the fibral flatness criterion and the fact that the formation of $G/\!/_\sigma G$ commutes with all base change \cite[Theorem 3.10]{Cotner-conn-comps}, we may assume $S = \Spec k$ for a field $k$. In this case, the result is \cite[Corollary 4.3.3, Corollary 5.3.3]{Xiao-Zhu}.
\end{proof}

Define the universal twisted centralizer $Z_{G,\sigma}$ be the closed $G$-subgroup scheme of $G \times_S G$ defined by the Cartesian diagram
\[
\begin{tikzcd}
    Z_{G,\sigma} \arrow[r] \arrow[d]
        &G \times_S G \arrow[d, "{(g, h) \mapsto (gh\sigma(g)^{-1}, h)}"] \\
    G \arrow[r, "\Delta"]
        &G \times_S G
\end{tikzcd}
\]

\begin{lemma}\label{lemma:flatness-of-fixed-point-scheme}
    If $G$ is semisimple and simply connected, then $Z_{G,\sigma} \times_G G^{\sigma\textrm{-}\rm{reg}} \to G^{\sigma\textrm{-}\rm{reg}}$ is a finitely presented flat commutative group scheme.
\end{lemma}

\begin{proof}
    Note that the map $G \times_S G^{\sigma\textrm{-}\rm{reg}} \to G \times_S G$, $(g, h) \mapsto (gh\sigma(g)^{-1}, h)$ factors through the closed subscheme $G^{\sigma\textrm{-}\rm{reg}} \times_{G/\!/_\sigma G} G^{\sigma\textrm{-}\rm{reg}}$, which is smooth by Lemma~\ref{lemma:steinberg-map-flat}. Furthermore, by fibral flatness, miracle flatness, and the definition of $\sigma$-regularity, the factored map $G \times_S G^{\sigma\textrm{-}\rm{reg}} \to G^{\sigma\textrm{-}\rm{reg}} \times_{G/\!/_\sigma G} G^{\sigma\textrm{-}\rm{reg}}$ is faithfully flat. Thus from the Cartesian diagram
    \[
    \begin{tikzcd}
        Z_{G,\sigma} \times_G G^{\sigma\textrm{-}\rm{reg}} \arrow[r] \arrow[d]
        &G \times_S G^{\sigma\textrm{-}\rm{reg}} \arrow[d, "{(g, h) \mapsto (gh\sigma(g)^{-1}, h)}"] \\
    G^{\sigma\textrm{-}\rm{reg}} \arrow[r, "\Delta"]
        &G^{\sigma\textrm{-}\rm{reg}} \times_{G/\!/_\sigma G} G^{\sigma\textrm{-}\rm{reg}}
    \end{tikzcd}
    \]
    we conclude the finite presentation and flatness claims. For commutativity, we may use the Existence and Isomorphism Theorems to reduce to the case $S$ is an integral scheme, in which case flatness allows us to pass to the generic fiber to assume $S$ is the spectrum of a field. In this case, \cite[Lemma 5.2.14]{Xiao-Zhu} shows that $G^{\sigma\textrm{-}\rm{reg}}$ admits a dense open subscheme $\Omega$ such that the $\Omega$-group scheme $Z_{G,\sigma} \times_G \Omega$ has torus fibers. In particular, $Z_{G,\sigma} \times_G \Omega$ is commutative, and the same is therefore true of $Z_{G,\sigma}$ by flatness.
\end{proof}

\begin{lemma}\label{lemma:twisted-centralizer-in-B}
    Suppose $S = \Spec k$ and $G$ is semisimple and simply connected. If $b \in B(k)$ is $\sigma$-regular, then $Z_{G,\sigma}(b) \subset B$.
\end{lemma}

\begin{proof}
    By Lemma~\ref{lemma:flatness-of-fixed-point-scheme}, the map $Z_{G,\sigma} \times_{G} B^{\sigma\textrm{-}\rm{reg}} \to B^{\sigma\textrm{-}\rm{reg}}$ is flat, so it is enough to prove the lemma for $b$ lying in a dense open subscheme of $B$. By \cite[Lemma 5.2.14]{Xiao-Zhu}, we may therefore assume $b \in B(k)$ is $\sigma$-regular semisimple. In that case, we have $b = b_0t\sigma(b_0)^{-1}$ for some $t \in T(k)$ and $b_0 \in B(k)$ and $Z_{G,\sigma}(b) = b_0Z_{G,\sigma}(t)b_0^{-1}$. By \cite[Remark 5.2.12]{Xiao-Zhu}, we have $Z_{G,\sigma}(t) = T^\sigma \subset B$, and the result follows.
\end{proof}

\subsection{Smoothness}\label{ss:smooth}

From now on, we will assume $S = \Spec k$. In this section, we will consider smoothness properties of twisted centralizers $Z_{G,\sigma}(g)$.

\begin{lemma}\label{lemma:reg-unip-centralizer}
    If $u \in U(k)$ is $\sigma$-regular, then $Z_{B,\sigma}(u) = Z(G)^\sigma \cdot Z_{U,\sigma}(u)$ and $Z_{U,\sigma}(u)$ is smooth.
\end{lemma}

\begin{proof}
    Let $\Delta_1, \dots, \Delta_n$ be the $\sigma$-orbits in $\Delta$, and for each $i$ choose $\alpha_i \in \Delta_i$. By \cite[Propositions 5.2.4(4), 5.2.19(2)]{Xiao-Zhu}\footnote{The part of \cite[Proposition 5.2.4(4)]{Xiao-Zhu} to which we refer here (i.e., the first sentence) is unaffected by the comments at the beginning of this section.}, we may assume $u = \prod_{i=1}^n x_{\alpha_i}(1)$. Fix a $k$-algebra $R$, and let $b \in Z_{B,\sigma}(u)(R)$. Write $b = sv$ for $s \in T(R)$ and $v \in U(R)$. We have
    \[
    u = svu\sigma(v)^{-1}\sigma(s)^{-1} = s\sigma(s)^{-1} \cdot \sigma(s)vu\sigma(v)^{-1}\sigma(s)^{-1},
    \]
    so it follows that $s = \sigma(s)$ and $u = svu\sigma(v)^{-1}s^{-1}$. Let $U'$ be the closed subgroup of $U$ generated by the root groups corresponding to non-simple roots, and say $v \equiv \prod_{\alpha \in \Delta} x_\alpha(v_\alpha) \pmod{U'}$ for $v_\alpha \in R$. Since $s = \sigma(s)$, we have
    \begin{equation}\label{eqn:sigma-regular-1}
    svu\sigma(v)^{-1}s^{-1} \equiv \prod_{i=1}^n \prod_{j=0}^{|\Delta_i| - 1} x_{\sigma^j\alpha_i}(\alpha_i(s)(\delta_{j0} + v_{\sigma^j\alpha_i} - v_{\sigma^{j-1}\alpha_i})) \pmod{U'}.
    \end{equation}
    We will now see this implies $\alpha_i(s) = 1$ for all $i$, which is enough to establish the first claim. If $|\Delta_i| = 1$, i.e., $\sigma\alpha_i = \alpha_i$, then (\ref{eqn:sigma-regular-1}) shows immediately that $\alpha_i(s) = 1$. If $|\Delta_i| > 1$, then (\ref{eqn:sigma-regular-1}) and the definition of $u$ combine to show $0 = \alpha_i(s)(v_{\sigma^j\alpha_i} - v_{\sigma^{j-1}\alpha_i})$ for $j \not\equiv 0 \pmod{|\Delta_i|}$. It follows that $v_{\alpha_i} = v_{\sigma^j\alpha_i}$ for all $j$. Moreover, we have $1 = \alpha_i(s)(1 + v_\alpha - v_{\sigma^{-1}\alpha}) = \alpha_i(s)$, as desired. Smoothness of $Z_{U,\sigma}(u)$ follows from the first (correct) statement of \cite[Lemma 5.2.5]{Xiao-Zhu}.
\end{proof}

The proof of the following lemma is very similar to that of \cite[Lemma 3.11]{DHKM}.

\begin{lemma}\label{lemma:smooth-center-ss}
    If $\tau$ is a semisimple automorphism of $G$, then $G^\tau$ is a reductive group and the quotient $Z((G^\tau)^0)/(Z(G)^\tau)^0$ is smooth.
\end{lemma}

\begin{proof}
    Semisimplicity of $\tau$ means that the algebraic subgroup $\Sigma$ of $\Aut_{G/k}$ generated by $\tau$ is of multiplicative type. Thus we may pass from $G$ to $G^{\Sigma^0}$ (a Levi subgroup of $G$) to assume that $\tau$ is of finite order not divisible by $\ell$. In that case, note that $G^\tau$ is reductive by \cite[Theorem 2.1]{Prasad-Yu-finite}. For the second claim, let $G_{\rm{ad}}$ denote the adjoint group of $G$, so $\tau$ induces an automorphism (which we will also denote by $\tau$) of $G_{\rm{ad}}$. The map $(G^\tau)^0 \to ((G_{\rm{ad}})^\tau)^0$ is surjective, and there is a short exact sequence
    \[
    1 \to Z(G)^\tau \cap (G^\tau)^0 \to Z((G^\tau)^0) \to Z((G_{\rm{ad}}^\tau)^0) \to 1,
    \]
    so we may pass from $G$ to $G_{\rm{ad}}$ to assume that $G$ is of adjoint type.

    We may assume $k = \ov{k}$. If $\chara k = 0$, then this follows from Cartier's theorem, so assume $\chara k = p > 0$. Let $\bG$ be the split reductive group scheme over $W(k)$ with special fiber $G$, and lift $\tau$ to a semisimple $W(k)$-automorphism of $\bG$ of order prime to $p$. Note that $(\bG^\tau)^0$ is a reductive $W(k)$-group scheme by, for instance, \cite[Corollary 3.7]{Booher-Tang}, so $Z((\bG^\tau)^0)$ is of multiplicative type. It therefore suffices to show that the component group of $Z((\bG^\tau)^0)_{W(k)[1/p]}$ is of order prime to $p$. By \cite[Theorem 7.5]{SteinbergEndomorphisms} and \cite[Lemma 5]{Springer-twisted}, we can write $\tau = c_t\sigma$ for an automorphism $\sigma$ preserving a pinning $(B, T, \{X_\alpha\})$ and $t \in T^\sigma(\ov{W(k)[1/p]})$, where $c_t$ denotes the $t$-conjugation map on $\bG_{\ov{W(k)[1/p]}}$. Note that $\sigma$ and $t$ are both of order prime to $p$, so the result follows from \cite[Theorem 3.11]{DM-quass}.
\end{proof}

\begin{lemma}\label{lemma:smooth-center-pinning}
    Suppose $G$ is simple and simply connected and $b \in B(k)$. If $(b,\sigma) = \tau_{\rm{ss}} \tau_{\rm{u}}$ is the Jordan decomposition of $(b,\sigma)$ in $G \rtimes \langle \sigma \rangle$, then $Z(G^{\tau_{\rm{ss}}})^{\sigma}/Z(G)^\sigma$ is smooth.
\end{lemma}

\begin{proof}
    We may assume $k = \ov{k}$.
    Note that $\sigma$ is of prime order. If $\sigma$ is of order not divisible by $\chara k$, then the result follows from Lemma~\ref{lemma:smooth-center-ss}, so we may assume $\sigma$ is of order $\chara k$. Lemma~\ref{lemma:smooth-center-ss} implies that $Z(G^{\tau_{\rm{ss}}})^0 = Z(L)^0$ for a $\sigma$-stable Levi subgroup $L$, so we need only show $Z(L)^\sigma/Z(G)^\sigma$ is smooth.
    
    By \cite[Lemma 6.7]{ALRR}, the subgroup $Z(G)^\sigma$ (resp.\ $Z(L)^\sigma$) is the center of $G^\sigma$ (resp.\ $L^\sigma$), and by \cite[Lemma 6.5]{ALRR} if $G^\sigma$ is smooth then $L^\sigma$ is a Levi subgroup of a parabolic subgroup of $G^\sigma$. By \cite[Theorem 5.1, Remark 5.3]{ALRR}, if either $\chara k \neq 2$ or $G$ is not of type $A_{2n}$ for $n \geq 1$, then $G^\sigma$ is smooth, so
    \[
    Z(L)^\sigma/Z(G)^\sigma = Z(L^\sigma)/Z(G^\sigma)
    \]
    is smooth.
    
    From now on, we consider the case that $G \cong \SL_{2n+1}$, $\chara k = 2$, and $\sigma \neq 1$. We may assume that $\sigma$ is the standard pinning-preserving automorphism $\sigma(g) = J(g^\top)^{-1}J^{-1}$, where $J$ is the anti-diagonal matrix with alternating $1$s and $-1$s. Let $T$ be the diagonal torus of $G$ and let $\Delta$ be the set of simple roots corresponding to the upper-triangular Borel $B$ of $G$. If $P$ is a parabolic subgroup of $G$ containing $B$ and $L$ is a $\sigma$-stable Levi of $P$ containing $T$, then $L$ corresponds to an ordered partition $(m_1, \dots, m_r)$ of $n+1$ and a direct calculation shows $Z(L)^\sigma \cong \bGm^{r-1}$, proving the lemma in this case.
\end{proof}


\begin{theorem}\label{theorem:smoothness-of-twisted-centralizers}
    If $G = Z \times \sD G$, where $Z$ is the maximal central torus of $G$, and $g \in G(k)$ is $\sigma$-regular, then $Z_{G,\sigma}(g)/Z(G)^\sigma$ is smooth.
\end{theorem}

\begin{proof}
    Note that if $g = (z, g_0)$, then $Z_{G,\sigma}(g) = Z^\sigma \times Z_{\sD G,\sigma}(g_0)$, so we may pass from $G$ to $\sD G$ to assume that $G$ is semisimple and simply connected. Next, we reduce to the case that $G$ is simple. By decomposing $G$ into simple factors, we may assume that $\sigma$ permutes the simple factors of $G$ transitively. Thus we may choose an isomorphism $G \cong H^m$, where $H$ is simple and $\sigma = \sigma_0 \circ \sigma_1$, where $\sigma_1(h_1, \dots, h_m) = (h_2, \dots, h_m, h_1)$ and $\sigma_0$ is the identity on each simple factor of $G$ except the last, on which it is pinning-preserving. Note that if $g = (h_1, \dots, h_m)$, then a point $(x_1, \dots, x_m)$ lies in $Z_{G,\sigma}(g)$ if and only if $(x_1, \dots, x_m) = (h_1x_2h_1^{-1}, \dots, h_{m-1}x_mh_{m-1}^{-1}, h_m\sigma_0(x_1)h_m^{-1})$. This is the case if and only if
    \[
    (x_1, \dots, x_m) = (x_1, h_1^{-1}x_1h_1^{-1}, \dots, h_{m-1}^{-1} \cdots h_1^{-1}x_1h_1 \cdots h_{m-1}),
    \]
    and $h_1 \cdots h_m \sigma_0(x_1) h_m^{-1} \cdots h_1^{-1} = x_1$. This shows that $Z_{G,\sigma}(g)/Z(G)^\sigma \cong Z_{H,\sigma_0}(h_1 \cdots h_m)/Z(H)^{\sigma_0}$, so we may reduce from $(G, \sigma, g)$ to $(H, \sigma_0, h_1 \cdots h_m)$ to assume that $H$ is simple.

    By \cite[Lemma 5(i)]{Springer-twisted}, we may assume $g \in B(k)$, in which case Lemma~\ref{lemma:twisted-centralizer-in-B} shows that $Z_{G,\sigma}(g) \subset B$. If $(g,\sigma) = \tau_{\rm{ss}} \tau_{\rm{u}}$ is the Jordan decomposition of $(g,\sigma)$ in $G \rtimes \langle\sigma\rangle$, then we have $Z_{G,\sigma}(g) = (G^{\tau_{\rm{ss}}})^{\tau_{\rm{u}}}$, so by Lemma~\ref{lemma:smooth-center-pinning} we may pass from $(G, (g,\sigma))$ to $(G^{\tau_{\rm{ss}}}, \tau_{\rm{u}})$ to assume that $(g,\sigma)$ is unipotent, i.e., $g$ is unipotent. (Note that since $G$ is semisimple and simply connected, the group $G^{\tau_{\rm{ss}}}$ is connected by \cite[Theorem 8.1]{SteinbergEndomorphisms}.) In this case, the result is Lemma~\ref{lemma:reg-unip-centralizer}.
\end{proof}

\begin{remark}
    Theorem~\ref{theorem:smoothness-of-twisted-centralizers} is proven in \cite[Corollary 3.5]{Integral-Springer} when $\sigma = 1$. There, the argument proceeds by first establishing the result when $g \in G(k)$ is unipotent, and then relying on a somewhat delicate argument of Steinberg involving root systems, showing that if $\chara k \nmid |\pi_1(\sD G)|$ and $t \in G(k)$ is semisimple, then $\chara k \nmid |\pi_1(\sD Z_G(t))|$. To deal with general $\sigma$, the analogous claim is that, if the image of $\ker \pi_{\sD G}$ in $Z(G_{\mathrm{sc}})_\sigma$ is smooth, $(Z \cap \sD G)^0 \subset (1-\sigma)Z \cap (1-\sigma)Z(\sD G)$, and $\tau$ is a semisimple automorphism of $G$ commuting with $\sigma$, then the image of $\ker \pi_{\sD G^\tau}$ in $Z((G^\tau)_{\mathrm{sc}})_\sigma$ is smooth and $(Z(G^\tau)^0_{\rm{red}} \cap \sD G^\tau)^0 \subset (1-\sigma)Z(G^\tau)^0_{\rm{red}} \cap (1-\sigma)Z(\sD G^\tau)$. We were unable to show this directly, which is why the argument given here is somewhat different. Still, it would be interesting to have a direct root-system-theoretic proof of this statement.
\end{remark}

\section{Stability for twisted conjugacy classes}\label{section:stability}

Throughout this section, let $G$ be a connected reductive group over an algebraically closed field $k$ and let $\sigma\colon G \to G$ be a $k$-automorphism preserving a pinning\footnote{As mentioned in Section~\ref{section:twisted}, this assumption is just for convenience.} $(B, T, \{X_\alpha\})$, where $B = TU$. Let $\chi_\sigma = \chi_{G, \sigma}\colon G \to G/\!/_\sigma G$ be the twisted GIT quotient. If $\sigma = 1$, we write $\chi \coloneqq \chi_\sigma$. Let $\Phi$ denote the set of roots of $(G, T)$, and let $W$ denote the Weyl group of $\Phi$.

Let $C_{g,\sigma}$ denote the twisted conjugacy class of $g$ in $G$, a reduced locally closed subscheme. Let $C_g = C_{g,1}$. The main aim of this section is to compute $\Stab_{Z(G)}(C_{g,\sigma})$ for certain elements $g \in G(k)$. If $g$ is regular, this amounts to understanding $\Stab_{Z(G)}(\chi_\sigma^{-1}(\chi_\sigma(g))_{\sigma\textrm{-}\rm{reg}})$, which we will handle in Section~\ref{section:Steinberg}. We will describe $\Stab_{Z(\sD G)^0}(C_{g,\sigma})$ first in the case $\sigma = 1$, and then bootstrap from this to describe the stabilizer of $C_g$ in $Z(G)$ for all regular $g$. In fact, we will focus on understanding unipotent $g$ first. Using this, we will be able to describe the stabilizer of $C_{g,\sigma}$ for $\sigma$-regular $g$, as well as $C_g$ for all $g$ when $G = \SL_n$.

\subsection{Steinberg fibers}\label{section:Steinberg}

By \cite[Corollary 5.3.5]{Xiao-Zhu}, if $G$ is semisimple and simply connected then (twisted) Steinberg fibers $\chi_\sigma^{-1}(\chi_\sigma(g))$ are simply the closures of $\sigma$-regular twisted conjugacy classes in $G$. Thus in the regular case, the study of closures of twisted conjugacy classes can be aided by the Steinberg map $\chi_\sigma$.

If $\sigma = 1$, then $\chi_\sigma$ is well-behaved only if $\sD G$ is simply connected; otherwise, it is (for instance) rarely flat, and it always has non-reduced fibers (see Corollary~\ref{cor:unip-small-char}). A similar phenomenon remains true for general $\sigma$ (but the precise conditions on $G$ are slightly harder to describe; see Corollary~\ref{cor:asc-adjoint-smooth}). Nonetheless, the underlying reduced subschemes of the fibers of $\chi_\sigma$ are still useful. The main results of this section are Theorem~\ref{theorem:stabilizer-of-Steinberg} and Theorem~\ref{theorem:torsor-identification}, which combine with the following lemma to give a rather precise description of the fibers of $\chi_\sigma$ in terms of the corresponding fibers for a sort of ``universal cover" (to be defined below) for $G$.

\begin{lemma}\label{lemma:stability-equivalences}
    Suppose $G = Z \times \sD G$, where $\sD G$ is simply connected, let $g \in G(k)$ be $\sigma$-regular, and let $\mu \subset Z(G)$ be a closed subscheme. The following are equivalent:
    \begin{enumerate}
        \item $C_{g,\sigma}$ is stable under $\mu$-multiplication
        \item $\ov{C}_{g,\sigma}$ is stable under $\mu$-multiplication,
        \item $B \cap \ov{C}_{g,\sigma}$ is stable under $\mu$-multiplication,
        \item $T \cap \ov{C}_{g,\sigma}$ is stable under $\mu$-multiplication,
        \item there exists $x \in \ov{C}_{g,\sigma}(k)$ such that $\mu x \subset \ov{C}_{g,\sigma}$.
    \end{enumerate}
    In particular, if $\pi\colon G \to G'$ is a central isogeny such that $\ker \pi$ is $\sigma$-stable and $\mu = \Stab_{\ker \pi}(\ov{C}_{g,\sigma})$, then the map $\ov{C}_{g,\sigma} \to \ov{C}_{\pi(g),\sigma'}$ is a $\mu$-torsor, where $\sigma'$ is the automorphism of $G'$ induced by $\sigma$.
\end{lemma}

\begin{proof}
    The implication $(1) \Rightarrow (2)$ follows from the fact that $C_{g,\sigma}$ is schematically dense in $\ov{C}_{g,\sigma}$. The implications $(2) \Rightarrow (3) \Rightarrow (4)$ are clear because $T$ contains $\mu$. The implication $(4) \Rightarrow (5)$ follows from the fact that $T \cap \ov{C}_{g,\sigma}$ is nonempty by \cite[Lemma 5(iii)]{Springer-twisted}. Thus we need only show $(5) \Rightarrow (1)$. By \cite[Lemma 5(i)]{Springer-twisted}, we may assume $g, x \in B(k)$ and that the $T$-components of $g$ and $x$ are equal. By \cite[Corollary 5.3.5]{Xiao-Zhu}, we have $\ov{C}_{g,\sigma} = \chi_\sigma^{-1}(\chi_\sigma(g))$, so (5) says that $\chi_\sigma|_{\mu x}$ is constant. Since $\chi_\sigma|_B$ is invariant under $U$-multiplication, it follows that $\chi_\sigma|_{\mu g}$ is constant and thus (1) follows.

    For the final claim, note that $\pi_g\colon \ov{C}_{g,\sigma} \to \ov{C}_{g',\sigma'}$ is $\mu$-invariant. If $x' \in \ov{C}_{g',\sigma'}(k)$, then since (2) and (5) are equivalent we have $\pi_g^{-1}(x') \cong \Transp_{\ker \pi}(x, \ov{C}_{g,\sigma}) = \mu$, where $x \in \ov{C}_{g,\sigma}(k)$ is any lift of $x'$. Thus the natural map $\mu \times \ov{C}_{g,\sigma} \to \ov{C}_{g,\sigma} \times_{\ov{C}_{g',\sigma'}} \ov{C}_{g,\sigma}$ given by $(z, x) \mapsto (zx, x)$ is an isomorphism by the fibral isomorphism criterion (applied to the structure morphisms to $\ov{C}_{g',\sigma'}$), so $\pi_g$ is a $\mu$-torsor.
\end{proof}

\begin{cor}\label{cor:normal}
	If $g \in G(k)$ is $\sigma$-regular, then $C_{g,\sigma}$ is smooth and $\ov{C}_{g,\sigma}$ is normal and lci.
\end{cor}

\begin{proof}
	If $G$ is a torus, then this is clear: indeed, $\ov{C}_{g,\sigma} = C_{g,\sigma}$ is smooth in this case. If $G$ is semisimple and simply connected, then this is \cite[Corollary 5.3.3, Corollary 5.3.5]{Xiao-Zhu}. In general, let $Z$ be the maximal central torus of $G$ and let $\widetilde{G} = Z \times \sD G$, so $\sigma$ induces a $k$-automorphism of $\widetilde{G}$ such that the multiplication map $\pi\colon \widetilde{G} \to G$ is $\sigma$-equivariant. If $\widetilde{g} \in \widetilde{G}(k)$ lifts $g$, then $\widetilde{g}$ is clearly $\sigma$-regular, so Lemma~\ref{lemma:stability-equivalences} shows that the map $\ov{C}_{\widetilde{g},\sigma} \to \ov{C}_{g,\sigma}$ is a torsor for a finite flat $k$-group scheme, and similarly for $C_{\widetilde{g},\sigma} \to C_{g, \sigma}$. The result therefore follows from \cite[Corollary 23.3, Theorem 23.7, Corollary 23.9]{Matsumura}.
\end{proof}

\begin{remark}\label{rmk:stability-non-reg}
    In the implication $(5) \Rightarrow (1)$ of Lemma~\ref{lemma:stability-equivalences}, it is important that $g$ be $\sigma$-regular. For example, consider the element
    \[
    g = \begin{pmatrix}
        1 & & & \\
          &1 & & \\
          & &-1 &1 \\
          & & &-1
    \end{pmatrix}
    \]
    of $\SL_4(k)$, where $\chara k \neq 2$. Note that $C_g$ is not invariant under multiplication by $-1$, but if $t$ is the semisimple part of $g$ then $C_t$ is invariant under multiplication of $-1$.

\end{remark}

We begin now by studying stability results for regular conjugacy classes in the usual sense; this will ultimately be combined with results of \cite{Springer-twisted} to deduce stability for twisted conjugacy classes in general.

\begin{lemma}\label{lemma:stabilizing-prime-power}
    Let $\ell \neq \chara k$ be prime, and suppose $k = \ov{k}$. If $\mu$ is the $k$-group scheme of $\ell$-power torsion in $Z(\sD G)$, then there is some semisimple $t \in G(k)$ of $\ell$-power order such that $\mu = \Stab_{Z(G)}(C_t)$.
\end{lemma}

\begin{proof}
    It is straightforward to reduce to the case that $G$ is semisimple and simply connected. By \cite[Lemma 9.2]{SteinbergEndomorphisms}, if $t \in G(k)$ is semisimple with image $\ov{t} \in G_{\rm{ad}}(k)$, then $\Stab_{Z(G)}(C_t) \cong \pi_0(Z_{G_{\rm{ad}}}(\ov{t}))$, so it is equivalent to find $\ov{t} \in G_{\rm{ad}}(k)$ of $\ell$-power order such that $\pi_0(Z_{G_{\rm{ad}}}(\ov{t})) \cong \mu$.
    
    By \cite[Theorem 9.1]{SteinbergEndomorphisms}, there is some semisimple $\ov{t}_0 \in G_{\rm{ad}}(k)$ such that $\pi_0(Z_{G_{\rm{ad}}}(\ov{t}_0)) \cong \mu$. Let $M$ be the closed subgroup scheme of $G_{\rm{ad}}$ generated by $\ov{t}_0$. Note that $M$ is of multiplicative type, so its torsion points are dense; since $Z_{G_{\rm{ad}}}(\ov{t}) \subset Z_{G_{\rm{ad}}}(\ov{t}_0)$ for all $\ov{t} \in M(k)$, upper semicontinuity of fiber dimension applied to the universal centralizer implies that there is an open set $\Omega \subset M$ containing $\ov{t}_0$ such that $Z_{G_{\rm{ad}}}(\ov{t})^0 = Z_{G_{\rm{ad}}}(\ov{t}_0)^0$ for all $\ov{t} \in \Omega(k)$. Let $T$ be a maximal $k$-torus of $G$ containing $M$ and let $W$ is the Weyl group of $(G, T)$, so every map $\Stab_W(\ov{t}) \to \pi_0(Z_{G_{\rm{ad}}}(\ov{t}))$ is surjective by \cite[Lemma 2.14]{SteinbergTorsion}. If $w \in W$ is an element not fixing $\ov{t}_0$, then the fixed point scheme for the action of $w$ on $T$ is a closed $k$-subscheme not containing $t_0$; thus by shrinking $\Omega$, we may assume that $Z_{G_{\rm{ad}}}(\ov{t}) = Z_{G_{\rm{ad}}}(\ov{t}_0)$ for all $\ov{t} \in \Omega(k)$. By density of the torsion points, it follows that there is some $\ov{t}_1 \in M(k)$ of finite order such that $Z_{G_{\rm{ad}}}(\ov{t}_1) = Z_{G_{\rm{ad}}}(\ov{t}_0)$.
    
    Write the order of $\ov{t}_1$ as $\ell^m n$, where $\ell \nmid n$, and let $\ov{t} = \ov{t}_1^n$, so $\ov{t}$ is of $\ell$-power order. We have
    \[
    Z_{G_{\rm{ad}}}(\ov{t}_1) = Z_{Z_{G_{\rm{ad}}}(\ov{t})}(\ov{t}_1^{\ell^m})
    \]
    and an exact sequence
    \[
    1 \to \pi_0(Z_{Z_{G_{\rm{ad}}}(\ov{t})^0}(\ov{t}_1^{\ell^m})) \to \mu \to \pi_0(Z_{G_{\rm{ad}}}(\ov{t})).
    \]
    By \cite[Corollary 2.16(b)]{SteinbergTorsion}, the order of $\pi_0(Z_{Z_{G_{\rm{ad}}}(\ov{t})^0}(\ov{t}_1^{\ell^m}))$ is prime to $\ell$ and thus $1$ by choice of $\mu$. Moreover, the same reference shows that $\pi_0(Z_{G_{\rm{ad}}}(\ov{t})) \subset Z(G)$ is of order a power of $\ell$, hence must be equal to $\mu$.
\end{proof}

\begin{lemma}\label{lemma:Springer-iso-6.4}
    If $\sD G$ is simply connected, then $\Stab_{Z(G)}(\chi^{-1}(\chi(1))) = Z(\sD G)^0$.
\end{lemma}

\begin{proof}
    We may assume $k = \ov{k}$ and $G = \sD G$. First note that $\Stab_{Z(G)}(\chi^{-1}(\chi(1))) \subset Z(\sD G)^0$ since $\chi^{-1}(\chi(1))(k)$ consists of unipotent elements and every element of $Z(G)(k)$ is semisimple. Thus we need only show the reverse inclusion.
    
    Let $p = \chara k$, and let $A$ be a DVR of mixed characteristic $(0, p)$ with residue field $k$ and fraction field $K$. Let $\cG$ be a reductive $A$-group scheme with special fiber $G$. By Lemma~\ref{lemma:stabilizing-prime-power}, there is a semisimple element $t \in \cG(K)$ of $p$-power order such that $\Stab_{Z(\cG)_K}(C_t)$ is the group $\mu$ of $p$-power torsion in $Z(\cG)_K$, i.e., the generic fiber of $Z(\cG)^0$. Note that, after passing to a finite extension of $K$, the element $t$ is $\cG(K)$-conjugate to an element of $\cG(A)$: indeed, $t$ is contained in $T(K)$ for a maximal $K$-torus $T \subset \cG_K$, which (after passing to a finite extension of $K$ and conjugating) we may assume is the generic fiber of a maximal $A$-torus $\cT$ of $\cG$ (since maximal $A$-tori of $\cG$ exist Zariski-locally by \cite[Expos\'e XIV, Corollaire 3.20]{SGA3II}). Since any torsion element of $T(K)$ necessarily lies in $\cT(A)$, we have $t \in \cG(A)$. Being of $p$-power order, the specialization $t_s \in G(k)$ is unipotent. By Lemma~\ref{lemma:stability-equivalences}, it follows that $\mu = \Stab_{Z(\cG)_K}(\chi^{-1}(\chi(t))_K)$. Since $\chi^{-1}(\chi(t))$ is $A$-flat with special fiber $\chi^{-1}(\chi(1))_k$, the result follows.
\end{proof}

\begin{remark}
    If $\chara k = p$ and $G = \SL_p$, then the proof of Lemma~\ref{lemma:Springer-iso-6.4} boils down to the observation that if $t = \diag(1, \zeta_p, \zeta_p^2, \dots, \zeta_p^{p - 1}) \in \SL_p(\bZ_p[\zeta_p])$, then $\chi^{-1}(\chi(t))$ is stabilized by all of $Z(G)$, and $\chi^{-1}(\chi(t))_{\bF_p} = \chi^{-1}(\chi(1))_{\bF_p}$.
\end{remark}

We now move on to understanding $\Stab_{Z(G)}(\chi^{-1}(\chi(g)))$ for general elements $g$. If $t \in T(k)$, let $W_t$ be the subgroup of $W$ consisting of those $w$ such that $wtw^{-1} = t$.

\begin{lemma}\label{lemma:omega-cap-Steinberg}
    Suppose $\sD G$ is simply connected. If $t \in T(k)$, then
    \[
    T \cap \chi^{-1}(\chi(t)) = \coprod_{w \in W/W_t} w(T \cap \chi_{Z_G(t)}^{-1}(\chi_{Z_G(t)}(t)))w^{-1}.
    \]
\end{lemma}

\begin{proof}
    We may and do assume $k = \ov{k}$ and $G$ is semisimple. We have a commutative diagram
    \[
    \begin{tikzcd}
    T \arrow{d}[swap]{\chi_{T,t}} \arrow[rd, "{\chi_{T}}"]
        & \\
    T/W_t \arrow[r, "\pi"]
        &T/W
    \end{tikzcd}
    \]
    By the Chevalley--Steinberg theorem \cite[Corollary 6.4]{SteinbergEndomorphisms} $\chi_{T,t}$ is the Steinberg map associated to $Z_G(t)$ (because $Z_G(t)$ has Weyl group $W_t$), we must show
    \[
    \chi_{T}^{-1}(\chi_{T}(t)) = \coprod_{w \in W/W_t} w\chi_{T,t}^{-1}(\chi_{T,t}(t))w^{-1}.
    \]
    The inclusion $(\supset)$ is evident, and since both sides are finite it is enough to show that the degree of the left side is at most the degree of the right side. By \cite[Lemma 6.3]{SteinbergReg}, the quotient $T/W$ is smooth, so miracle flatness implies that $\chi_T$ is finite flat of degree $|W|$, so $\chi_T^{-1}(\chi_T(t))$ is of degree $|W|$. Moreover, $\chi_{T,t}$ is finite of degree $|W_t|$. Two conjugates $w\chi_{T,t}^{-1}(\chi_{T,t}(t))w^{-1}$ and $w'\chi_{T,t}^{-1}(\chi_{T,t}(t))w'^{-1}$ are disjoint unless $w = w'$, so $\coprod_{w \in W/W_t} \chi_{T,t}^{-1}(\chi_{T,t}(wtw^{-1}))$ is of degree $\geq |W/W_t| \cdot |W_t| = |W|$, and we are done.
\end{proof}

\begin{theorem}\label{theorem:stabilizer-of-Steinberg}
    If $\sD G$ is simply connected and $t \in G(k)$ is semisimple, then for all $x \in \chi^{-1}(\chi(t))(k)$, we have
    \[
    \Stab_{Z(\sD G)^0}(\chi^{-1}(\chi(t))) = \sD Z_G(t) \cap Z(\sD G)^0
    \]
    as subschemes of $Z(\sD G)^0$.
\end{theorem}

\begin{proof}
    If we write $t = z_0t_0$ for $z_0 \in Z(G)(k)$ and $t_0 \in \sD G(k)$, then $\chi^{-1}(\chi(t)) = z_0 \cdot \chi_{\sD G}^{-1}(\chi_{\sD G}(t_0))$, so we may pass from $(G, t)$ to $(\sD G, t_0)$ to assume that $G$ is semisimple. Fix a $k$-algebra $R$. By Lemma~\ref{lemma:stability-equivalences} and Lemma~\ref{lemma:omega-cap-Steinberg}, translation by $z \in Z(G)^0(R)$ preserves $\chi^{-1}(\chi(t))$ if and only if it preserves $w(T \cap \chi_{Z_G(t)}^{-1}(\chi_{Z_G(t)}(t)))w^{-1}$ for all $w \in W$. Since $z$-translation commutes with the $W$-action, it is equivalent that $z$-translation preserves $T \cap \chi_{Z_G(t)}^{-1}(\chi_{Z_G(t)}(t))$. By centrality of $t$ in $Z_G(t)$, we have
    \[
    \chi_{Z_G(t)}^{-1}(\chi_{Z_G(t)}(t)) = t \cdot \chi_{Z_G(t)}^{-1}(\chi_{Z_G(t)}(1)) = t \cdot \chi_{\sD Z_G(t)}^{-1}(\chi_{\sD Z_G(t)}(1)).
    \]
    Now Lemma~\ref{lemma:Springer-iso-6.4} implies that $T \cap \chi_{\sD Z_G(t)}^{-1}(\chi_{\sD Z_G(t)}(1))$ is stable under $z$-translation if and only if $z \in Z(\sD Z_G(t))^0(R) = (\sD Z_G(t) \cap Z(\sD G)^0)(R)$.
\end{proof}

Theorem~\ref{theorem:stabilizer-of-Steinberg} gives a complete description of the stabilizer of $\chi^{-1}(\chi(t))$ in $Z(\sD G)^0$. Since $\Stab_{Z(G)}(\chi^{-1}(\chi(t))) \subset Z(\sD G)$, \cite[Theorem 9.1]{SteinbergEndomorphisms} shows that if $\sD G$ is simply connected and $t \in G(k)$ is semisimple, then the stabilizer of $t$ in $Z(G)_{\rm{red}}$ is isomorphic to $\pi_0(Z_{G^{\rm{ad}}}(\ov{t}))$, where $G^{\rm{ad}}$ is the adjoint quotient of $G$ and $\ov{t}$ is the image of $t$ in $G^{\rm{ad}}$.

The following corollary complements \cite[Lemma 4.4]{Integral-Springer}.

\begin{cor}\label{cor:unip-small-char}
Let $p = \chara k > 0$.
\begin{enumerate}
	\item\label{item:chi-unipotent-variety-non-reduced} If $t \in G(k)$ is semisimple, then $\chi^{-1}(\chi(t))$ is of generic multiplicity $\geq |\pi_1(\sD(Z_G(t)^0))[p^\infty]|\cdot |\pi_0(Z_G(t))|$. In particular, if $p$ divides $|\pi_1(\sD G)|$, then $\chi^{-1}(\chi(1))$ is generically non-reduced.
    \item\label{item:springer-resolution-generically-torsor} Let $B$ be a Borel $k$-subgroup of $G$ with unipotent radical $U$ and let $\widetilde{\sU}_G = G \times^B U$ be the Springer resolution of $\sU_G^{\rm{var}}$. If $\pi_G: \widetilde{G} \to \sD(G)$ is the universal cover, then the Springer resolution $\widetilde{\sU} \to \sU_G^{\rm{var}}$ is a $(\ker \pi_G)^0$-torsor over $\sU_{G, \rm{reg}}$.
\end{enumerate}
\end{cor}

\begin{proof}
We can reduce to the case that $G$ is semisimple and $k$ is algebraically closed. Let $\pi_G: \widetilde{G} \to G$ be the universal cover, and let $\widetilde{\chi}\colon \widetilde{G} \to \widetilde{G}/\!/\widetilde{G}$ be the Steinberg morphism. For (\ref{item:chi-unipotent-variety-non-reduced}), let $A$ be a DVR of mixed characteristic $(0, p)$ with residue field $k$ and fraction field $K$ and let $\bG$ be the split semisimple group scheme over $A$ whose special fiber is $G$. Let $\pi: \widetilde{\bG} \to \bG$ be the universal cover of $\bG$, so the GIT quotient map $\widetilde{\chi}\colon \widetilde{\bG} \to \widetilde{\bG}/\!/\widetilde{\bG}$ is flat with geometrically reduced fibers. Lift $t$ to an element $\bt \in \bG(A)$ such that $\bt_K$ is strongly regular semisimple, and pass to a faithfully flat extension of $A$ to assume there is a lift $\widetilde{\bt} \in \widetilde{\bG}(A)$ of $\bt$. If $\sC_\bt$ is the schematic closure of the generic fiber of $\chi^{-1}(\chi(\bt))$, then by Lemma~\ref{lemma:stability-equivalences}, Theorem~\ref{theorem:stabilizer-of-Steinberg}, and \cite[Theorem 4.1(1)]{Lee-adjoint}, the map $\widetilde{\chi}^{-1}(\widetilde{\chi}(\widetilde{\bt})) \to \sC_\bt$ is generically an isomorphism, and it is a $(\ker \pi \cap \sD Z_{\widetilde{G}}(\widetilde{t}))^0 \times \pi_0(Z_G(t))$-torsor on the special fiber. By \cite[Corollary 3.2(1)]{Integral-Springer}, the order of $\pi_1(\sD Z_{\widetilde{G}}(\widetilde{t}))$ is not divisible by $p$. Because $\sD Z_{\widetilde{G}}(\widetilde{t}) \to \sD(Z_G(t)^0)$ is a central isogeny, it follows that $(\ker \pi \cap \sD Z_{\widetilde{G}}(\widetilde{t}))^0$ is the identity component of the kernel of the universal cover of $\sD(Z_G(t)^0)$. Since $(\sC_\bt)_k \subset \chi^{-1}(\chi(t))$, Lemma~\ref{lemma:mult-flat-cover} proves the claim. 

For (\ref{item:springer-resolution-generically-torsor}), recall that the Springer resolution $\widetilde{\sU}_{\widetilde{G}} \to \sU_{\widetilde{G}}$ is an isomorphism over $\sU_{\widetilde{G}, \rm{reg}}$ by \cite[\S 6]{Steinberg-Desingularization}. Since the unipotent radical of a Borel is unchanged after arbitrary central isogeny, the definition of the Springer resolution shows that $\widetilde{\sU}_{\widetilde{G}} \to \widetilde{\sU}_G$ is an isomorphism. Now we have a commutative diagram
\[
\begin{tikzcd}
\widetilde{\sU}_{\widetilde{G}} \arrow[r] \arrow[d]
    &\widetilde{\sU}_G \arrow[d] \\
\sU_{\widetilde{G}} \arrow[r]
    &\sU_G^{\rm{var}}
\end{tikzcd}
\]
By Lemma~\ref{lemma:stability-equivalences} and Theorem~\ref{theorem:stabilizer-of-Steinberg}, the lower horizontal morphism is a $(\ker \pi_G)^0$-torsor. Combined with the previous remarks, it follows that the right vertical morphism is a $(\ker \pi_G)^0$-torsor over $\sU_{G, \rm{reg}}$, as claimed.
\end{proof}

\begin{remark}\label{remark:twisted-steinberg-fiber}
    In Corollary~\ref{cor:unip-small-char}(2), the inequality can be strict even when $\chara k = 0$. Indeed, \cite[Theorem 3]{Steinberg-Pittie} shows that if $k$ is algebraically closed and $G$ is simple then $G/\!/G$ is regular if and only if $G$ is simply connected or $G \cong \SO_{2n+1}$. For any other simple $G$, let $\pi\colon \widetilde{G} \to G$ be the universal cover, and assume for simplicity that $\pi$ is \'etale. The induced map $f\colon \widetilde{G}/\!/\widetilde{G} \to G/\!/G$ cannot be flat, because the source is regular but the target is not. Since $f$ is finite, it follows from commutative algebra that the fibers of $f$ cannot all have the same rank. If $\widetilde{g} \in \widetilde{G}(k)$ has image $x$ in $(G/\!/G)(k)$, then $f^{-1}(x)$ has precisely $[\ker\pi:\Stab_{\ker \pi}(C_{\widetilde{g}})]$ points, and these are permuted by the action of $\ker \pi$. Also $f$ is generically \'etale because $\pi$ is \'etale and $\chi$ is generically smooth \cite[Theorem 3.10]{Slodowy}, so $f$ is of degree $|\ker \pi|$. As already mentioned, there is some $\widetilde{g}$ with image $x$ in $(G/\!/G)(k)$ such that $f^{-1}(x)$ is of rank $> |\ker \pi|$. If $\widetilde{x}$ is the image of $\widetilde{g}$ in $(\widetilde{G}/\!/\widetilde{G})(k)$, then the component $C$ of $f^{-1}(x)$ supported at $\widetilde{x}$ is of rank $> |\Stab_{\ker \pi}(C_{\widetilde{g}})|$. Since $\pi$ is \'etale and $\widetilde{\chi}$ is flat with geometrically reduced fibers, Lemma~\ref{lemma:mult-flat-cover} shows that $C$ and $\chi^{-1}(x)$ have the same generic multiplicities, so the inequality of Corollary~\ref{cor:unip-small-char}(2) is strict.
    
    For an explicit example, let $G = \PGL_3$ over an algebraically closed field $k$, and let $t = \diag(1, \zeta, \zeta^2)$, where $\zeta$ is a primitive $3$rd root of unity (unless $\chara k = 3$, in which case let $\zeta = 1$). Corollary~\ref{cor:unip-small-char}(2) then says that the generic multiplicity of $\chi^{-1}(\chi(t))$ is $\geq 3$, while in fact a(n involved) calculation reveals the generic multiplicity to be $5$. One can also compute that if $G = \PGL_4$ and $t = \diag(1, i, -1, -i)$, where $i^2 = -1$, then $\chi^{-1}(\chi(t))$ is of generic multiplicity $10$. If instead $t = \diag(1, -1, 1, -1)$ and $\chara k \neq 2$, then $\chi^{-1}(\chi(t))$ is of generic multiplicity $3$. (I thank LoG(M) students Veer Agarwal, Linkun Ma, and Yan Yu for help calculating these latter two examples.)

    In general, I do not expect any simple method for computing the generic multiplicity of $\chi^{-1}(\chi(t))$ on the nose. In any case, for applications to Section~\ref{section:moduli}, this multiplicity is not very important, as the scheme $\chi^{-1}(\chi(t))$ will not naturally show up; a more relevant object is the scheme $\sC_{\bt}$ appearing in the proof of Corollary~\ref{cor:unip-small-char}(2).
\end{remark}

\subsection{Twisted Steinberg fibers}

Now we move on to the general case of twisted conjugacy classes of $\sigma$-regular elements. First suppose that $G$ is semisimple and simply connected. In that case, \cite{Springer-twisted} introduced a connected semisimple group $G_\sigma$ with maximal torus $T_\sigma$ and set of roots $\Phi^\sigma$, consisting of sums $\alpha_\cO \coloneqq \sum_{\alpha \in \cO} \alpha$, where $\cO$ ranges over all $\sigma$-orbits in $\Phi$.
By \cite[Proposition 1, Proposition 2]{Springer-twisted}, the group $G_\sigma$ is semisimple, and it has Weyl group $W^\sigma$. Thus by the usual Chevalley--Steinberg theorem \cite[Corollary 6.4]{SteinbergReg}, we have $G_\sigma/\!/G_\sigma \cong T_\sigma/W^\sigma$. (Note that both implicit actions of $W^\sigma$ on $T_\sigma$ are the same, as one can see by \cite[D\'efinition-Th\'eor\`eme 1.15]{DM-non-connexes}).

\begin{lemma}\label{lemma:stabilizer-not-in-type-A}
    Suppose that $G$ is semisimple and simply connected and there is no root $\alpha$ of $(G, T)$ such that $\langle \alpha^\vee, \sigma(\alpha)\rangle = -1$. If $q\colon T \to T_\sigma$ is the quotient and $\ov{q(t)}$ is the image of $q(t)$ in $G_\sigma^{\rm{ad}}$, then
    \[
    \Stab_{Z(G)^0}(\chi_\sigma^{-1}(\chi_\sigma(t))) = Z(G)^0 \cap q^{-1}(Z(G_\sigma) \cap \sD Z_{G_\sigma}(q(t))).
    \]
\end{lemma}

\begin{proof}
    Note that by \cite[Corollary 1(i)]{Springer-twisted}, the hypotheses imply that $G_\sigma$ is simply connected.\footnote{Case-checking shows that the hypothesis of this lemma implies that $G$ admits a simple factor isomorphic to $\SL_{2n+1}$ for some $n$.} As remarked above, we have $G_\sigma/\!/G_\sigma \cong T_\sigma/W^\sigma \cong G/\!/_\sigma G$ in such a way that the diagram
    \[
    \begin{tikzcd}
        T \arrow[r, "q"] \arrow[rd, "\chi_{G,\sigma}|_T"]
            &T_\sigma \arrow[d, "\chi_{G_\sigma}|_T"] \\
            &T_\sigma/W^\sigma
    \end{tikzcd}
    \]
    is commutative. In particular, we have
    \[
    \chi_{G,\sigma}|_T^{-1}(\chi_{G,\sigma}|_T(t)) = q^{-1}(\chi_{G_\sigma}|_T^{-1}(\chi_{G_\sigma}|_T(q(t)))).
    \]
    By Lemma~\ref{lemma:stability-equivalences} and Theorem~\ref{theorem:stabilizer-of-Steinberg}, the stabilizer of $T_\sigma \cap \chi_{G_\sigma}^{-1}(\chi_{G_\sigma}(q(t)))$ in $Z(G_\sigma)^0$ is $Z(G_\sigma)^0 \cap \sD Z_{G_\sigma}(q(t))$, so the stabilizer of $T \cap \chi_{G,\sigma}^{-1}(\chi_{G,\sigma}(t)) = q^{-1}(T_\sigma \cap \chi_{G_\sigma}^{-1}(\chi_{G_\sigma}(q(t))))$ is as claimed.
\end{proof}


\begin{lemma}\label{lemma:invariant-by-coinvariants}
    If $\sD G$ is simply connected and $g \in G(k)$, then the fiber $\chi_\sigma^{-1}(\chi_\sigma(g))$ is stable under multiplication by $(1-\sigma)Z(G)$.
\end{lemma}

\begin{proof}
    Let $R$ be a $k$-algebra, and suppose $s \in Z(G)(R)$. Because $\chi_\sigma$ is invariant under $\sigma$-conjugation, it follows that
    \[
    \chi_\sigma(g) = \chi_\sigma(sg\sigma(s)^{-1}) = \chi_\sigma(s\sigma(s)^{-1}g),
    \]
    as desired.
\end{proof}

\begin{lemma}\label{lemma:stabilizer-in-type-A}
    Suppose that $G \cong \SL_{2n+1}^m$ for some $n \geq 1$ and $m \geq 1$, and $\sigma = \sigma_0 \circ \sigma_1$, where $\sigma_1$ is defined by $\sigma_1(h_1, \dots, h_m) = (h_2, \dots, h_m, h_1)$ and $\sigma_0$ acts trivially on every simple factor of $G$ except the last, on which it is not inner. If $g \in G(k)$, then $\chi_\sigma^{-1}(\chi_\sigma(g))$ is stable under $Z(G)$-multiplication.
\end{lemma}

\begin{proof}
    By Lemma~\ref{lemma:invariant-by-coinvariants}, it suffices to show that $(1-\sigma)$ induces an automorphism of $Z(G)$. Note that $\sigma_0$ acts on $Z(\SL_{2n+1}) \cong \mu_{2n+1}$ by inversion, so the map $1 - \sigma$ is given on $Z(G)$ by
    \[
    (1 - \sigma)(z_1, \dots, z_m) = (z_1 z_2^{-1}, \dots, z_{m-1} z_m^{-1}, z_m z_1).
    \]
    This shows that $(1-\sigma)Z(G)$ consists of those $(z_1, \dots, z_m) \in Z(G)$ such that $\prod_{i=1}^m z_i \in Z(G)^2$. Since $Z(G)$ is of odd order, we see $(1-\sigma)Z(G) = Z(G)$.
\end{proof}

\begin{theorem}\label{theorem:torsor-identification}
    Suppose $G = Z \times \sD G$, where $Z$ is a torus and $\sD G$ is simply connected, and write $\sD G = \prod_{i=1}^n G_i$ for $\sigma$-stable reductive subgroups $G_i \subset \sD G$ such that $\sigma$ permutes the simple factors of each $G_i$ transitively. If $t \in T(k)$, then 
    \[
    \Stab_{Z(G)^0}(\chi_\sigma^{-1}(\chi_\sigma(t))) = (1-\sigma)Z \times \prod_{i=1}^n (Z(G_i)^0 \cap q^{-1}(Z((G_i)_\sigma) \cap \sD Z_{(G_i)_\sigma}(q(t))),
    \]
    where as above $q\colon T \to T_\sigma$ is the quotient map.
\end{theorem}

\begin{proof}
    We may assume $k = \ov{k}$. Write $\sD G = \prod_{i=1}^m G_i$ as above. It is clear that we may pass from $G$ either to $Z$ or to some $G_i$ to assume that $G$ is semisimple and $\sigma$ permutes the simple factors of $G$ transitively. The case $G = Z$ follows from Lemma~\ref{lemma:invariant-by-coinvariants}. In the other case, we may choose an isomorphism $G \cong H^m$, where $H$ is simple, such that $\sigma = \sigma_0 \circ \sigma_1$, where $\sigma_1(h_1, \dots, h_m) = (h_2, \dots, h_m, h_1)$ and $\sigma_0$ is the identity on every simple factor of $G$ except the last. If $H \cong \SL_{2n+1}$ for some $n$ and $\sigma_0$ is not inner, then $Z(G) = (1-\sigma)Z(G) \subset q^{-1}(1)$ and the result is Lemma~\ref{lemma:stabilizer-in-type-A}. Otherwise, the hypotheses of Lemma~\ref{lemma:stabilizer-not-in-type-A} apply and the result again follows.
\end{proof}



Now we use Theorem~\ref{theorem:torsor-identification} to strengthen Theorem~\ref{theorem:smoothness-of-twisted-centralizers}. These results will be used in Section~\ref{section:moduli} to deal with multiplicities in the moduli space of L-parameters valued in ramified groups. First, we need a lemma.

\begin{lemma}\label{lemma:twisted-centralizer-isogeny}
    Consider a commutative diagram
    \[
    \begin{tikzcd}
        G' \arrow[r, "\sigma'"] \arrow[d, "f"]
            &G' \arrow[d, "f"] \\
        G \arrow[r, "\sigma"]
            &G
    \end{tikzcd}
    \]
    where $f$ is a central isogeny and $\sigma$, $\sigma'$ are automorphisms. Let $g' \in G'(k)$, and let $g = f(g')$. Then
    \begin{align*}
    Z_{G,\sigma}(g)/f(Z_{G',\sigma'}(g')) \cdot Z(G)^\sigma
    &\cong f^{-1}(Z_{G,\sigma}(g))/Z_{G',\sigma'}(g') \cdot f^{-1}(Z(G)^\sigma) \\
    &\cong \im(\Stab_{Z(G')}(C_{g', \sigma'}) \cap \ker(f) \to Z(G')_{\sigma'}).
    \end{align*}
\end{lemma}

\begin{proof}
    The first isomorphism is evident from surjectivity of $f$. For the second, define a homomorphism $\pi\colon f^{-1}(Z_{G,\sigma}(g)) \to Z(G')_{\sigma'}$ by $\pi(x) = xg'\sigma(x)^{-1}(g')^{-1}$. The definitions show that the image of $\pi$ is equal to the image of $\Stab_{Z(G)}(C_{g', \sigma'}) \cap \ker(f)$, and it is simple to check that $\ker(\pi) = Z_{G',\sigma'}(g') \cdot f^{-1}(Z(G)^\sigma)$.
\end{proof}

\begin{definition}\label{def:almost-simply-connected}
    Let $\pi_G\colon \widetilde{G} \coloneqq Z \times G_{\mathrm{sc}} \to G$ be the natural map induced by multiplication, where $Z$ is the maximal central torus of $G$ and $G_{\mathrm{sc}}$ is the universal cover of $\sD G$. We will say that the pair $(G, \sigma)$ is \textit{almost simply connected} if $\ker \pi_G \subset (1-\sigma)Z(\widetilde{G})$, where $\sigma$ is used for the automorphism of $\widetilde{G}$ induced by the automorphism of $G$ of the same name.
\end{definition}

\begin{cor}\label{cor:asc-smooth}
    If $(G, \sigma)$ is almost simply connected and $\widetilde{g} \in \widetilde{G}(k)$ is $\sigma$-regular, then $\pi_G(Z_{\widetilde{G}, \sigma}(\widetilde{g})) = Z_{G, \sigma}(\pi_G(\widetilde{g}))$. In particular, $Z_{G, \sigma}(g)/Z(G)^\sigma$ is smooth for all $\sigma$-regular $g \in G(k)$.
\end{cor}

\begin{proof}
    The first claim is immediate from Lemma~\ref{lemma:invariant-by-coinvariants} and Lemma~\ref{lemma:twisted-centralizer-isogeny}. The second follows from Theorem~\ref{theorem:smoothness-of-twisted-centralizers}.
\end{proof}

\begin{cor}\label{cor:asc-adjoint-smooth}
    If $(G, \sigma)$ is almost simply connected, then $\chi_\sigma\colon G \to G/\!/_\sigma G$ is flat. If $G^{\sigma\textrm{-reg}}$ is the open $\sigma$-regular locus in $G$, then $\chi_\sigma|_{G^{\sigma\textrm{-reg}}}$ is smooth.
\end{cor}

\begin{proof}
    If $G = Z \times G'$, where $G'$ is simply connected, then the claim follows from \cite[Corollary 4.3.3, Corollary 5.3.3]{Xiao-Zhu}. In general, the map $\pi_G\colon \widetilde{G} = Z \times G_{\rm{sc}} \to G$ above induces a map $\widetilde{G}/\!/_\sigma\widetilde{G} \to G/\!/_\sigma G$. This map is an isomorphism: indeed, we have
    \[
    \widetilde{G}/\!/_\sigma \widetilde{G} \cong (\widetilde{G}/\!/_\sigma Z(\widetilde{G}))/\!/_\sigma G \cong (G/\!/_\sigma Z(G))/\!/_\sigma G,
    \]
    where the first isomorphism follows from the fact that $G$ is a central quotient of $\widetilde{G}$ and the second follows from the fact that $\widetilde{G}/\!/_\sigma Z(\widetilde{G}) = \widetilde{G}/(1-\sigma)Z(\widetilde{G})$, which is a central quotient of $G$. Since the map $\pi_G$ restricts to a $\ker(\pi_G)$-torsor $\widetilde{G}^{\sigma\textrm{-reg}} \to G^{\sigma\textrm{-reg}}$, the result follows.
\end{proof}

The following two lemmas will be used in Section~\ref{section:moduli}.

\begin{lemma}\label{lemma:almost-simply-connected}
    The map $\pi_G$ factors uniquely as $\widetilde{G} \to \widetilde{G}_\sigma \xrightarrow[]{\pi_{G,\sigma}} G$, where $(\widetilde{G}_\sigma, \sigma)$ is almost simply connected, $\ker \pi_{G,\sigma} \subset Z(\widetilde{G}_\sigma)^\sigma$, and $\ker \pi_{G,\sigma} \cap (1-\sigma)Z(\widetilde{G}_\sigma) = 1$.
\end{lemma}

\begin{proof}
    Let $\widetilde{G}_\sigma = \widetilde{G}/(\ker \pi_G \cap (1-\sigma)Z(\widetilde{G}))$, so there is a factorization as in the lemma statement. Note that $\sigma$ induces an automorphism of $\widetilde{G}_\sigma$ (which we also call $\sigma$) because the kernel of $\widetilde{G} \to \widetilde{G}_\sigma$ is $\sigma$-stable. Moreover, $\ker \pi_{G,\sigma} = \ker \pi_G/(\ker \pi_G \cap (1-\sigma)Z(\widetilde{G}))$, so $\ker \pi_{G,\sigma} \subset Z(\widetilde{G}_\sigma)^\sigma$. The final equality is clear, and uniqueness is also clear.
\end{proof}

\begin{lemma}\label{lemma:transp-fppf}
    Let $S$ be a scheme, let $G$ be a reductive $S$-group scheme, and let $\sigma$ be a pinning-preserving $S$-automorphism of $G$ such that the pair $(G, \sigma)$ is almost simply connected.\footnote{This can be defined in precisely the same way as Definition~\ref{def:almost-simply-connected}, and it is equivalent to requiring that the $S$-fibers of $(G, \sigma)$ are almost simply connected.} If $x, y \in G(S)$ have $\sigma$-regular fibers and satisfy $\chi_\sigma(x) = \chi_\sigma(y)$, then there is an fppf map $S' \to S$ and $g \in G(S')$ such that $gx\sigma(g)^{-1} = y$.
\end{lemma}

\begin{proof}
    Let $G^{\sigma\textrm{-reg}}$ denote the open subscheme of $G$ consisting of $\sigma$-regular points of $G$. By the fibral flatness criterion and Corollary~\ref{cor:asc-adjoint-smooth}, the preimage $\chi_\sigma^{-1}(\chi_\sigma(x)) = \chi_\sigma^{-1}(\chi_\sigma(y))$ is $S$-smooth. By \cite[Corollary 5.3.5(1)]{Xiao-Zhu}, each fiber of $\chi_\sigma^{-1}(\chi_\sigma(x))^{\sigma\textrm{-reg}} \coloneqq \chi_\sigma^{-1}(\chi_\sigma(x)) \cap G^{\sigma\textrm{-reg}}$ is a single $G$-orbit. It follows from Corollary~\ref{cor:asc-smooth} and the fibral flatness criterion that the map $G \times \chi_\sigma^{-1}(\chi_\sigma(x))^{\sigma\textrm{-reg}} \to (\chi_\sigma^{-1}(\chi_\sigma(x))^{\sigma\textrm{-reg}})^2$, $(g, z) \mapsto (gz\sigma(g)^{-1}, z)$ is smooth and surjective, from which the result follows.
\end{proof}

\subsection{General conjugacy classes}\label{ss:general-conj}

In this section, we will describe the stabilizer of the conjugacy class $C_u = C_{G, u}$ of a unipotent element $u \in \SL_n(k)$ under multiplication by the central $\mu_n$. In order to facilitate a future complete answer to Question~\ref{question:main}, we will work for as long as possible with an arbitrary connected reductive $k$-group $G$.

\begin{lemma}\label{lemma:transverse-orbit-centralizer}
    Let $t \in G(k)$ be semisimple. If $C$ is the $G$-conjugacy class of $t$, then there is a neighborhood $\Omega$ of $t$ in $G$ such that $\Omega \cap C \cap Z_G(t) = \{t\}$ scheme-theoretically.
\end{lemma}

\begin{proof}
    It is enough to show that the tangent spaces $\rm{Tan}_t C$ and $\rm{Tan}_t Z_G(t)$ of $C$ and $Z_G(t)$ at $t$ intersect trivially. If one identifies $\rm{Tan}_t G$ with $\fg$ via translation, then the map $\fg \to \rm{Tan}_t C \subset \fg$ induced from the (smooth) $t$-orbit map $G \to C$ is given by $X \mapsto (\Ad_{t^{-1}} - 1)X$. Thus $\rm{Tan}_t C$ can be identified with $\im(\Ad_{t^{-1}} - 1)$. Since $\rm{Tan}_t Z_G(t)$ is identified with $\ker(\Ad_{t^{-1}} - 1)$ and $t$ is semisimple, the result follows.
\end{proof}

\begin{lemma}\label{lemma:semisimple-conjugacy-class-cap-T}
    If $T$ is a maximal $k$-torus of $G$ and $t \in T(k)$ has $G$-conjugacy class $C$, then $C \cap T = \{wtw^{-1}\colon w \in W\}$ scheme-theoretically.
\end{lemma}

\begin{proof}
    It is a standard fact that $C(k) \cap T(k) = \{wtw^{-1}\colon w \in W\}$. Since $T \subset Z_G(wtw^{-1})$ for all $w \in W$, Lemma~\ref{lemma:transverse-orbit-centralizer} shows that the intersection $C \cap T$ is transverse, proving the lemma.
\end{proof}

\begin{lemma}\label{lemma:stability-passes-to-jordan}
    Suppose $t, u \in G(k)$ are commuting elements such that $t$ is semisimple and $Z_G(u)_{\rm{red}}$ is a normal $k$-subgroup of $Z_G(u)$.\footnote{For general group schemes, this is not automatic: for instance, if $H = (\bZ/p)^\times \ltimes \mu_p$, then $H_{\rm{red}} = (\bZ/p)^\times$ is not normal in $H$. However, it holds if $Z_G(u)/Z(G)$ is smooth, for instance if $G \cong \SL_n$.}
    If $\mu \subset Z(\sD G)^0$ is a $k$-subgroup scheme such that $C_u$ is stable under multiplication by $\mu$, then the same is true of $C_{tu}$ and $C_{Z_G(t), u}$.
\end{lemma}

\begin{proof}
    We may and do assume $k = \ov{k}$. To show that $C_{Z_G(t), u}$ is $\mu$-stable, it is enough to show that, if $R$ is the coordinate ring of $\mu$, then $(C_{Z_G(t), u})_R$ is stable under multiplication by $\mu(R)$. Thus let $R$ be an artin local $k$-algebra with residue field $k$, and let $z \in \mu(R)$. Since $C_u$ is stable under multiplication by $\mu$, we may extend $R$ to assume that there is some $g \in G(R)$ such that $gug^{-1} = zu$. Since $\mu(k) = \{1\}$, the special fiber $g_s \in G(k)$ centralizes $u$. Thus by passing from $g$ to $g_s^{-1}g$, we may assume $g_s = 1$. Note also that $gtg^{-1} \in Z_G(u)(R)$.

    Let $M$ be the closed $k$-subgroup of $G$ generated by $t$, which is of multiplicative type since $t$ is semisimple. Note that $M_R$ is the smallest closed $R$-subgroup scheme of $G_R$ containing the section $t_R$, from which it follows that $gM_R g^{-1}$ is contained in $Z_G(u)_R$. Since $Z_G(u)/Z_G(u)_{\rm{red}}$ is infinitesimal and $M$ is smooth, the induced map $gM_R g^{-1} \to (Z_G(u)/Z_G(u)_{\rm{red}})_R$ has trivial special fiber. By \cite[Expos\'e IX, Corollaire 3.5]{SGA3II}, it follows that $gM_Rg^{-1} \subset (Z_G(u)_{\rm{red}})_R$.
    
    Since $Z_G(u)_{\rm{red}}$ is smooth, $R$ is strictly henselian, and $g_s = 1$, it follows from \cite[Expos\'e IX, Corollaire 3.3bis]{SGA3II} that there is some $h \in Z_G(u)_{\rm{red}}(R)$ with $h_s = 1$ such that $hgM_R g^{-1}h^{-1} = M_R$. Thus $hgtg^{-1}h^{-1} \in M(R)$, so $hgtg^{-1}h^{-1} = t$ by Lemma~\ref{lemma:semisimple-conjugacy-class-cap-T} and the fact that $h_s g_s = 1$. Since $hgug^{-1}h^{-1} = zu$ and $hg \in Z_G(t)(R)$, we see that $C_{tu}$ and $C_{Z_G(t), u}$ are both stable under multiplication by $\mu$.
\end{proof}


\begin{prop}\label{prop:stability-type-a}
    Let $G = \SL_n$ and let $u \in G(k)$ be unipotent, corresponding to the partition $(m_1, \dots, m_\ell)$ of $n$. If $p^r$ is the largest power of $p$ dividing each $m_i$, then the stabilizer of $C_u$ in $Z(G)$ is the central $\mu_{p^r}$.
\end{prop}

\begin{proof}
    Let $\mu$ be the stabilizer of $C_u$ in $Z(G)$, and note that $\mu \subset Z(G)^0$ because $u$ is unipotent. Let $L$ be a Levi subgroup of $G$ such that $u$ is regular in $L(k)$, and let $t \in L(k)$ be a semisimple element such that $L = Z_G(t)$. By Lemma~\ref{lemma:Springer-iso-6.4}, $C_u$ is stable under multiplication by $Z(G)^0 \cap \sD L = \mu_{p^r}$. Moreover, since $tu$ is regular, Theorem~\ref{theorem:stabilizer-of-Steinberg} shows that the stabilizer of $C_{tu}$ is $\mu_{p^r}$. These observations combine with Lemma~\ref{lemma:stability-passes-to-jordan} to show that $\mu = \mu_{p^r}$.
\end{proof}

\begin{remark}
    The argument of the proposition can be used to give a description of the stabilizer of $C_g$ in $Z(G)$ for general $g \in G(k)$ when $G = \SL_n$; in view of Lemma~\ref{lemma:stability-passes-to-jordan}, it suffices to describe the stabilizer in $Z(G)_{\rm{red}}$. If $g = tu$ is the Jordan decomposition, then this stabilizer is a subgroup of the stabilizer of $C_t$, consisting of those elements preserving $C_{Z_G(t), u}$. In view of Remark~\ref{rmk:stability-non-reg}, it seems unlikely that there is a cleaner description than this.
\end{remark}

\begin{remark}\label{remark:non-regular}
	We conclude this section with some speculations on a method to obtain a complete answer to Question~\ref{question:main} for general $G$. A natural attempt is to mimic the proof of Theorem~\ref{theorem:stabilizer-of-Steinberg} and try to find a lift $x$ of $g$ to $\cG(\cO)$, where $\cO$ is a discrete valuation ring with residue field $k$ and fraction field $K$ of characteristic $0$, $\cG$ is a reductive $\cO$-group scheme with special fiber $G$, and the $Z(\cG_K)$-stabilizer of $C_{x_K}$ is ``large". To make this strategy work, one would need to find $x$ such that $C_g$ is an open subscheme of the special fiber of $\overline{C_{x_K}}$: if $C_g$ is merely a locally closed subscheme (even an open subscheme of the underlying reduced subscheme) of this special fiber, then showing that $\overline{C_{x_K}}$ is stable under $\mu$-translation for a $K$-subgroup scheme $\mu \subset Z(\cG_K)$ would not a priori imply anything about the action of $\overline{\mu}^0$ on $C_g$. In particular, to make this strategy work one would need to choose $x$ to be \textit{pure} in the sense of \cite[Definition 3.8]{Cotner-flat}, i.e., $\dim Z_G(g) = \dim Z_{\cG_K}(x_K)$.
	
	However, this is not enough: \cite[Remark 5.5]{Cotner-flat} shows that if $x \in \cG(\cO)$ is a pure section then it can happen that $Z_{\cG}(x)$ is not flat, and in this case $\overline{C_{x_K}}$ does not have generically reduced generic fiber. Indeed, if $Z_{\cG}(x)_0$ denotes the schematic closure of $Z_{\cG_K}(x_K)$, then $Z_{\cG}(x)_0$ is flat and the map $\cG/Z_{\cG}(x)_0 \to \overline{C_{x_K}}$, $h \mapsto hxh^{-1}$ is quasi-finite, so Lemma~\ref{lemma:multiplicity-behavior} and Zariski's main theorem show that if $Z_{\cG}(x)_0 \neq Z_{\cG}(x)$ then $(\ov{C_{x_K}})_k$ is not reduced. Thus to use this strategy, the first goal would be to lift $g$ to $x$ such that $Z_{\cG}(x)$ is flat. To show that $\overline{C_{x_K}}$ has reduced special fiber, one would then like to apply Lemma~\ref{lemma:multiplicity-behavior} to the orbit map $\cG/Z_{\cG}(x) \to \overline{C_{x_K}}$, which would require showing that this map is finite. Using the valuative criterion of properness, that amounts to giving a positive answer to the following question.
\end{remark}

\begin{question}\label{question:finiteness}
	Let $A$ be a complete discrete valuation ring with residue field $k$ and fraction field $K$, let $\cG$ be a reductive $A$-group scheme, and let $x, y \in \cG(A)$ be two pure sections such that
	\begin{enumerate}
		\item $x_k$ is $\cG(k)$-conjugate to $y_k$,
		\item $x_K$ is $\cG(K)$-conjugate to $y_K$.
	\end{enumerate}
	Does it follow that there is a finite extension $A \subset A'$ of discrete valuation rings such that $x_{A'}$ is $\cG(A')$-conjugate to $y_{A'}$?
\end{question}

We can answer Question~\ref{question:finiteness} positively when $x$ is either fiberwise regular, fiberwise semisimple, or fiberwise unipotent in good characteristic. The first two cases can be extracted from the results of \cite{Integral-Springer} and \cite{Cotner-flat} with some work; for the final case, see \cite[Theorem 5.11]{Cotner-flat}. Using the Jordan decomposition of \cite[Theorem 4.1]{Cotner-flat}, one can generalize the latter two results slightly. Note that the answer to Question~\ref{question:finiteness} is negative if one does not assume that $x$ and $y$ are pure; this can already be seen for the elements $x = \begin{pmatrix}
	1 &p \\ 0&1
\end{pmatrix}$ and $y = \begin{pmatrix}
	1 &p^2 \\ 0&1
\end{pmatrix}$ in $\GL_2(\bZ_p)$. Question~\ref{question:finiteness} can be generalized in an evident way to actions of $G$ on any smooth $A$-scheme, and we do not have any counterexamples in this generality either.

It seems plausible to the author that this strategy could be used to describe the $Z(\sD G)^0$-stabilizers of unipotent elements $g$ which come from the Bala--Carter theorem, by lifting $g$ to a section with semisimple generic fiber using an associated cocharacter. However, it is not clear to the author how to carry out the above steps even in this case, and in bad characteristic (the only interesting remaining case) not all unipotent elements arise in this way. Since the answer to Question~\ref{question:main} is not essential to the applications which follow, we will pursue this no further.

\section{Counterexamples}\label{section:counterexamples}

In this section we show that the hypotheses in \cite[Theorem 3.10]{Integral-Springer} and \cite[Theorem 5.1]{Integral-Springer} are necessary. In Section~\ref{subsection:regular-unipotent-centralizers}, we will show that \cite[Corollary 3.11]{Integral-Springer} fails for centralizers of regular unipotent elements over a field whenever $|\pi_1(\sD G)|$ is zero in the field. In Section~\ref{ss:nilp}, we will observe several pathologies with the nilpotent scheme in small characteristic. In Section~\ref{subsection:no-springer-homeomorphism}, we will show that there isn't even a $G$-equivariant finite $k$-morphism $\sN_G^{\rm{var}} \to \sU_G^{\rm{var}}$ when $|\pi_1(\sD G)|$ is zero in the field. Finally, in Section~\ref{subsection:no-springer-iso}, we will study the Picard groups of the unipotent and nilpotent varieties and show that they are different when $|\pi_1(\sD G)|$ is zero in the field (at least under the assumption that the latter is normal).

\subsection{Centralizers of regular unipotent elements over a field}\label{subsection:regular-unipotent-centralizers}

Our main aim in this section is to prove the following theorem, which will both help us to show that the hypotheses in \cite[Theorem 5.1]{Integral-Springer} are optimal and answer a question left open in \cite[Remark 2.7]{BRR}. We will only deal with usual (non-twisted) orbits and centralizers, since the twisted case requires considerable extra notation but is of dubious additional interest.

\begin{theorem}\label{theorem:regular-unipotent-centralizer}
Let $G$ be a connected reductive group over a field $k$ of characteristic $p \geq 0$ and let $u \in G(k)$ be a regular unipotent element contained in a (unique) Borel subgroup $B$ of $G$. The following are equivalent.
\begin{enumerate}
    \item\label{item:unipotent-centralizer-smooth} $Z_G(u)/Z(G)$ is smooth.
    \item\label{item:unipotent-centralizer-in-borel} $Z_G(u) \subset B$.
    \item\label{item:unipotent-centralizer-product} The multiplication morphism $Z(G) \times Z_U(u) \to Z_G(u)$ is an isomorphism, where $U$ is the unipotent radical of $B$.
    \item\label{item:unipotent-centralizer-fundamental-group} $p \nmid |\pi_1(\sD(G))|$.
\end{enumerate}
Consider the following condition:
\begin{enumerate}[resume]
	\item\label{item:unipotent-centralizer-commutative} $Z_G(u)$ is commutative.
\end{enumerate}
The conditions (\ref{item:unipotent-centralizer-smooth})-(\ref{item:unipotent-centralizer-fundamental-group}) imply (\ref{item:unipotent-centralizer-commutative}), and the converse holds if $p$ is good for $G$.
\end{theorem}

\begin{remark}\label{remark:not-always-commutative}
It is stated (and proved) in various places in the group theory literature (see \cite{Springer-note} for good characteristic and \cite{Lou} or \cite[III, Corollary 1.16]{Springer-Steinberg} in general) that $Z_G(u)$ is always commutative whenever $u$ is a regular unipotent element of $G(k)$. This may make the statement of Theorem~\ref{theorem:regular-unipotent-centralizer} appear rather curious. Of course, in the classical references given above, the centralizer $Z_G(u)$ is identified with its set of $k$-points (if $k$ is algebraically closed); equivalently, the statement is that $Z_G(u)_{\rm{red}}$ is commutative. In Lemma~\ref{lemma:centralizer-kernel-smooth}, we will see that if $G$ is semisimple and $\pi_G: \widetilde{G} \to G$ is the universal cover then there is a short exact sequence
\[
1 \to Z_{\widetilde{G}}(\widetilde{u}) \to \pi_G^{-1}(Z_G(u)) \to (\ker \pi_G)^0 \to 1
\]
where $Z_{\widetilde{G}}(\widetilde{u})$ is commutative by \cite[Corollary 3.11]{Integral-Springer}. This gives rise to a short exact sequence
\[
1 \to Z_{\widetilde{G}}(\widetilde{u})/\ker \pi_G \to Z_G(u) \to (\ker \pi_G)^0 \to 1.
\]
In this sense, the only source of non-commutativity is infinitesimal, conforming to the known commutativity results. See Example~\ref{example:pgl2} for an explicit example. Moreover, by \cite[Lemma 3.4]{Integral-Springer}, $Z_U(u)$ is smooth and thus always commutative.
\end{remark}

\begin{proof}[Proof of Theorem~\ref{theorem:regular-unipotent-centralizer}]
First, we note that (\ref{item:unipotent-centralizer-fundamental-group}) $\Rightarrow$ (\ref{item:unipotent-centralizer-smooth}) follows from \cite[Lemma 3.4]{Integral-Springer}. For the remainder, we may and do assume that $k$ is algebraically closed. 

Next we prove (\ref{item:unipotent-centralizer-smooth}) $\Rightarrow$ (\ref{item:unipotent-centralizer-in-borel}). Since $B$ is the unique Borel $k$-subgroup of $G$ containing $u$ and $N_G(B) = B$, it follows that $Z_G(u)(k) = Z_B(u)(k)$. To prove that $Z_G(u)$ is contained in $B$, it is enough to show that the natural homomorphism $Z(G) \times Z_U(u) \to Z_G(u)$ is an isomorphism, and since both sides contain $Z(G)$ the five lemma shows it is enough to check this after modding out by $Z(G)$. But the map $Z_U(u) \to Z_G(u)/Z(G)$ is bijective on geometric points and has trivial kernel, so by smoothness of $Z_G(u)/Z(G)$ it must be an isomorphism.

The implication (\ref{item:unipotent-centralizer-in-borel}) $\Rightarrow$ (\ref{item:unipotent-centralizer-product}) follows from the first statement of \cite[Lemma 3.4]{Integral-Springer}. 

To prove (\ref{item:unipotent-centralizer-product}) $\Rightarrow$ (\ref{item:unipotent-centralizer-commutative}), it suffices to show that $Z_U(u)$ is commutative. For this, we may assume that $G$ is semisimple. If $G$ is simply connected, then this follows from \cite[Corollary 3.11]{Integral-Springer}. In general, if $f: G' \to G$ is a central isogeny of connected semisimple groups, $U'$ is the unipotent radical of $f^{-1}(B)$, and $u'$ is a unipotent element of $U'(k)$ lifting $u$, then the induced map $Z_{U'}(u') \to Z_U(u)$ is an isomorphism. Thus commutativity of $Z_U(u)$ is reduced to the settled simply connected case. 

It remains to show that (\ref{item:unipotent-centralizer-product}) $\Rightarrow$ (\ref{item:unipotent-centralizer-fundamental-group}) and that (\ref{item:unipotent-centralizer-commutative}) $\Rightarrow$ (\ref{item:unipotent-centralizer-fundamental-group}) when $\chara k$ is good for $G$.

\begin{lemma}\label{lemma:centralizer-kernel-smooth}
Suppose $G$ is semisimple and let $f: G' \to G$ be a central isogeny. Let $u' \in G'(k)$ be regular unipotent with image $u = f(u')$, and let $\nu\colon f^{-1}(Z_G(u)) \to \ker f$ be the $k$-homomorphism $g \mapsto gu'g^{-1}u'^{-1}$. There is a short exact sequence
\[
1 \to Z_{G'}(u') \to f^{-1}(Z_G(u)) \xrightarrow[]{\nu} (\ker f)^0 \to 1
\]
\end{lemma}

\begin{proof}
This is a direct consequence of Lemma~\ref{lemma:twisted-centralizer-isogeny} and Lemma~\ref{lemma:Springer-iso-6.4}.
\end{proof}

For the remainder, we may and do assume that $G$ is semisimple. If (\ref{item:unipotent-centralizer-product}) holds, then $\pi_G^{-1}(Z_G(u)) = Z_{\widetilde{G}}(\widetilde{u})$, where $\widetilde{u}$ is the unique unipotent lift of $u$ to $\widetilde{G}(k)$, and Lemma~\ref{lemma:centralizer-kernel-smooth} implies $(\ker \pi_G)^0 = 1$, i.e., $p \nmid |\pi_1(G)|$, proving (\ref{item:unipotent-centralizer-fundamental-group}).

For (\ref{item:unipotent-centralizer-commutative}) $\Rightarrow$ (\ref{item:unipotent-centralizer-fundamental-group}), assume that $p$ is good for $G$. We will show first (under the assumption that $Z_G(u)$ is commutative) that if $\pi_G: \widetilde{G} \to G$ is the universal cover, then $\pi_G^{-1}(Z_G(u))$ is commutative. 

Let $f: G' \to G$ be any central isogeny. By \cite[Expos\'e XVII, Th\'eor\`eme 7.2.1]{SGA3II}, since $Z_G(u)$ is commutative there is an isomorphism $Z_G(u) \cong M_0 \times Z_U(u)$ for some multiplicative type subgroup $M_0$ of $Z_G(u)$. If $M_0' = M_0/Z(G)$, then $M_0'$ is finite and connected: to see this, note that the inclusion $Z(G) \times Z_U(u) \to M_0 \times Z_U(u)$ is a nilpotent thickening because $Z_G(u)_{\rm{red}} = Z_B(u)_{\rm{red}}$ and $Z_B(u) = Z(G) \times Z_U(u)$ by \cite[Lemma 3.4]{Integral-Springer}. There is a short exact sequence
\[
1 \to Z(G') \to f^{-1}(Z_G(u)) \to M_0' \times Z_U(u) \to 1.
\]
By \cite[Expos\'e XVII, Proposition 7.1.1]{SGA3II}, the preimage $M_1$ of $M_0'$ in $f^{-1}(Z_G(u))$ is a normal subgroup of multiplicative type, and in particular it is commutative. Moreover, \cite[Expos\'e XVII, Th\'eor\`eme 6.1.1 A) iv)]{SGA3II} shows
\[
f^{-1}(Z_G(u)) \cong M_1 \rtimes Z_U(u).
\]
Since $p$ is good for $G$, the group $Z_U(u)$ is connected by \cite[Theorem 4.11]{Springer}. Since the Aut scheme of $M_1$ is \'etale, it follows that $f^{-1}(Z_G(u)) \cong M_1 \times Z_U(u)$ is commutative.

Now Lemma~\ref{lemma:centralizer-kernel-smooth} implies that the map $\nu\colon \pi_G^{-1}(Z_G(u)) \to \ker \pi_G$, $\nu(x) = x\widetilde{u}x^{-1}\widetilde{u}^{-1}$ is a flat homomorphism. Since $\pi_G^{-1}(Z_G(u))$ is commutative and contains $\widetilde{u}$, this map is trivial. Thus flatness implies $\ker \pi_G$ is \'etale, i.e., $p \nmid |\pi_1(G)|$.
\end{proof}

\begin{cor}\label{cor:non-flat-unip}
Let $A$ be a DVR of mixed characteristic $(0, p)$, and let $G$ be a reductive $A$-group scheme such that $p$ is good for $G$ and divides $|\pi_1(\sD(G))|$. If $u \in G(A)$ is fiberwise regular unipotent, then $Z_G(u)$ is not flat.
\end{cor}

\begin{proof}
By \cite[Corollary 3.11]{Integral-Springer}, $Z_G(u)_\eta$ is commutative; by Theorem~\ref{theorem:regular-unipotent-centralizer}, $Z_G(u)_s$ is not commutative. This cannot be the case for a flat group scheme.
\end{proof}

\subsection{Pathologies with the nilpotent variety}\label{ss:nilp}

We begin with the following weak analogue of Theorem~\ref{theorem:stabilizer-of-Steinberg} in type $\rm{A}$.

\begin{lemma}\label{lemma:nilpotent-variety-stable}
Let $G = \prod_{i=1}^m\SL_{n_i}$ over a field $k$ of characteristic $p > 0$. If $p^{e_i}$ is the largest power of $p$ dividing $n_i$, then $\sN_G$ is stable under addition by the scheme $\alpha\coloneqq \prod_{i=1}^m\ker [p^{n_i}]_{\fz_i}$, where $\fz_i$ is the center of $\mathfrak{sl}_{n_i}$ and $[p]: \mathfrak{sl}_{n_i} \to \mathfrak{sl}_{n_i}$ is the $p$-operation. If $\mu \subset Z(G)$ is of order divisible by $p$ and $\pi\colon G \to G' = G/\mu$ is the quotient map, then the natural map $\sN_G \to \sN_{G'}^{\rm{var}}$ is a torsor for $\Lie \mu \cap \alpha$. Furthermore, if $X$ is a regular nilpotent element of $\Lie G$, then there is a short exact sequence
\[
1 \to \pi(Z_G(X)) \to Z_G(\pi(X)) \to \Lie \mu \cap \alpha \to 1,
\]
where the second nontrivial map is given by $g \mapsto \Ad_g(X) - X$.
\end{lemma}

\begin{proof}
For the first two claims, one immediately reduces to the case $G = \SL_n$. By \cite[Lemma 4.11(3)]{Integral-Springer}, $\sN_G$ is cut out of $\mathfrak{sl}_n$ by the coefficients of the characteristic polynomial. Explicitly, if $X = (x_{ij})$ is an $n \times n$ matrix over some ring $R$ then $X$ lies in $\sN_G(R)$ if and only if $\det(tI - X) = t^n$, where $t$ is a formal variable. If $X \in \sN_G(R)$ and $a \in \fz(R)$, then we have
\[
\det(tI - (X + aI)) = \det((t-a)I - X) = (t-a)^n = (t^{p^e} - a^{p^e})^m
\]
if $n = p^e m$ with $p \nmid m$. This is equal to $t^n$ if and only if $a^{p^e} = 0$, i.e., $a \in \alpha(R)$. By \cite[Lemma 4.11(2)]{Integral-Springer}, $\sN_G \to \sN_{G'}^{\rm{var}}$ is surjective, so this calculation shows that it is an $\alpha$-torsor, as desired.

For the final claim, it is easy to check that $g \mapsto \Ad_g(X) - X$ is a homomorphism $Z_{G/\mu}(X) \to \alpha$ with kernel $\pi(Z_G(X))$. Furthermore, the morphism $G/\mu \to \fg$ with the same definition factors through a $Z_G(X)/\mu$-torsor $G/\mu \to \sN_G^{\rm{reg}} - X \subset \fg$. Since this map is surjective, there is a Cartesian diagram
\[
\begin{tikzcd}
    Z_{G/\mu}(X) \arrow[r] \arrow[d]
        &G/\mu \arrow[d] \\
    \alpha \arrow[r]
        &\sN_G^{\rm{reg}} - X
\end{tikzcd}
\]
which shows that the left vertical map is faithfully flat. This implies the lemma.
\end{proof}

\begin{lemma}\label{lemma:nilp-var-emb-dim}
    Let $k$ be a field of characteristic $p > 0$, and let $G$ be a connected semisimple $k$-group such that $p$ divides $|\pi_1(G)|$. Let $\chi\colon \fg \to \fg/\!/G$ be the natural quotient map. Then $\chi^{-1}(\chi(0))$ is a closed subscheme of an affine space of dimension $< \dim G$.
\end{lemma}

\begin{proof}
    Let $\pi\colon \widetilde{G} \to G$ be the universal cover of $G$, and let $\fz \subset \widetilde{\fg}$ be the (nonzero) kernel of $\Lie \pi$, so there is a short exact sequence
    \[
    0 \to \fz \to \widetilde{\fg} \to \fq \to 0
    \]
    of $G$-representations. This induces a short exact sequence
    \[
    0 \to \fq \to \fg \to \fr \to 0,
    \]
    where $\fr$ is of dimension $\dim \fz \geq 1$. Note that $G$ acts trivially on $\fr$: indeed, we may assume that $k$ is algebraically closed, so there are split reductive group schemes $\widetilde{\cG}$ and $\cG$ over the ring of Witt vector $W(k)$ such that $\widetilde{\cG}_k = \widetilde{G}$ and $\cG_k = G$. Moreover, $\pi$ occurs as the special fiber of a central isogeny $\widetilde{\cG} \to \cG$. Note that $\Lie \widetilde{\cG}$ and $\Lie \cG$ are both $W(k)$-lattices in $\Lie \cG_{W(k)[1/p]}$, so the Jordan--H\"older factors of the $G$-representations $\widetilde{\fg} \cong \Lie \widetilde{\cG} \otimes_{W(k)} k$ and $\fg \cong \Lie \cG \otimes_{W(k)} k$ are identical by the Brauer--Nesbitt theorem. Thus $\fr \cong \fz$ as $G$-representations, i.e., $G$ acts trivially on $\fr$. Since $\fr^*$ is a vector space of $G$-invariant functions on $\fg^*$, we have $\chi^{-1}(\chi(0)) \subset (\fr^*)^\perp$. Since $(\fr^*)^\perp$ is of dimension $\dim G - \dim \fz < \dim G$, we conclude.
\end{proof}

The following lemma should be contrasted with \cite[Theorem 4.1(1)]{Lee-adjoint}, which shows that for a reductive group scheme $G$ over a scheme $S$, the formation of the GIT quotient $G/\!/G$ commutes with arbitrary base change.

\begin{lemma}\label{lemma:nilp-scheme-base-change}
    Let $p$ be a prime number, and let $G$ be a semisimple group scheme over $\bZ_p$ such that $p$ divides $|\pi_1(G)|$. Then $(\fg/\!/G)_{\bF_p} \not\cong \fg_{\bF_p}/\!/G_{\bF_p}$, i.e., the formation of $\fg/\!/G$ does not commute with base change to $\bF_p$.
\end{lemma}

\begin{proof}
    Let $\chi\colon \fg \to \fg/\!/G$ and $\overline{\chi}\colon \fg_{\bF_p} \to \fg_{\bF_p}/\!/G_{\bF_p}$ be the natural quotient maps. It is enough to show that $\chi^{-1}(\chi(0))_{\bF_p} \neq \overline{\chi}^{-1}(\overline{\chi}(0))$. By Lemma~\ref{lemma:nilp-var-emb-dim}, the preimage $\overline{\chi}^{-1}(\overline{\chi}(0))$ embeds into an affine space of dimension $< \dim G$, and in particular the embedding dimension of $\overline{\chi}^{-1}(\overline{\chi}(0))$ is $< \dim G$ at every point. Thus it suffices to show that the embedding dimension of $\chi^{-1}(\chi(0))_{\bF_p}$ is $\geq \dim G$ at $0$.

    By \cite[Proposition 5.2.9, Theorem 9.1.4]{Alper-adequate}, the formation of the GIT quotient $\fg/\!/G$ commutes with flat base change, so $\chi^{-1}(\chi(0))_{\bQ_p}$ is the nilpotent variety of $G_{\bQ_p}$, which is the same as the nilpotent variety $\sN$ of $G_{\bQ_p}$. The embedding dimension of $\sN$ at $0$ is equal to $\dim G$: to see this, observe that $\sN$ is defined as a closed subscheme of $\fg_{\bQ_p}$ by homogeneous polynomials of degree $\geq 2$. By Lemma~\ref{lemma:emb-dim-ineq}, it follows that $\chi^{-1}(\chi(0))_{\bF_p}$ has embedding dimension $\geq \dim G$ at $0$. (In fact, it has embedding dimension exactly $\dim G$ since it is a subscheme of $\fg$.)
\end{proof}

The following corollary shows that the hypotheses for the smoothness result in \cite[Theorem 3.10(2)]{Integral-Springer} and the flatness result in \cite[Theorem 4.2.8]{Bouthier-Cesnavicius} are optimal.

\begin{cor}\label{cor:non-flat-nilp}
Let $A$ be a DVR of mixed characteristic $(0, p)$, and let $G$ be a reductive $A$-group scheme such that $p$ is a bad prime for $G$. If $X \in (\Lie G)(A)$ is regular nilpotent, then $Z_G(X)/Z(G)$ is not smooth. If $p$ is a torsion prime for $G$ or $p$ divides $|\pi_1(\sD(G))|$, then $Z_G(X)$ is not flat.
\end{cor}

\begin{proof}
For the first claim, it suffices to work over an algebraically closed field $k$ of characteristic $p > 0$. Using \cite[Theorem 2.6]{Springer}, one can show that if $p$ is a bad prime for $G$ and $X \in \Lie B$ for a Borel $k$-subgroup $B$ of $G$ with unipotent radical $U$, then $Z_U(X)$ is not smooth. Consequently, $Z_G(X)/Z(G)$ cannot be smooth. 

Now assume either that $p$ is a torsion prime for $G$ or $p$ divides $|\pi_1(\sD(G))|$. By passing to a quasi-finite local extension of $A$, we may choose a Borel $A$-subgroup $B$ of $G$ such that $X \in (\Lie B)(A)$. By \cite[Proposition 3.8]{Integral-Springer}, $Z_G(X)_\eta$ lies in $B_\eta$. Using \cite[Proposition 2.4, Theorem 2.6]{Springer} and Lemma~\ref{lemma:nilpotent-variety-stable}, one can show that $Z_G(X)_s$ does not lie in $B_s$, so $Z_G(X)$ cannot be flat.
\end{proof}

\subsection{Ruling out an equivariant dominant morphism}\label{subsection:no-springer-homeomorphism}

In \cite[Proposition 29]{Mcninch-subprincipal}, it is shown that if $G$ is a connected reductive group over an algebraically closed field $k$ of good characteristic $p$ (possibly with $p$ dividing $|\pi_1(\sD G)|$), then there is a $G(k)$-equivariant homeomorphism $\sN_G(k) \to \sU_G(k)$ (with respect to the Zariski topologies on both sides). This can be obtained from the diagram
\[
\begin{tikzcd}
\sU_{\widetilde{G}} \arrow[r, "\rho"] \arrow[d]
    &\sN_{\widetilde{G}} \arrow[d] \\
\sU_G
    &\sN_G
\end{tikzcd}
\]
where $\widetilde{G}$ is the universal cover of $\sD G$ and $\rho$ is a Springer isomorphism, since each arrow is a universal homeomorphism of schemes. However, as we will show, such a homeomorphism cannot be algebraic if $p$ divides $|\pi_1(\sD G)|$.

\begin{prop}\label{prop:no-dominant-map}
    Let $k$ be any field of characteristic $p > 0$, and let $G$ be a connected reductive group over $k$ such that $p$ is good for $G$ and $p$ divides $|\pi_1(\sD G)|$. Then there is no $G$-equivariant dominant $k$-morphism $\sN_G^{\rm{var}} \to \sU_G^{\rm{var}}$ or $\sU_G^{\rm{var}} \to \sN_G^{\rm{var}}$.
\end{prop}

\begin{proof}
    Recall that $\sU_G^{\rm{var}}$ contains a dense open $G$-orbit $\sU_G^{\rm{reg}}$, and similarly $\sN_G^{\rm{var}}$ contains a dense open $G$-orbit $\sN_G^{\rm{reg}}$. For dimension reasons, a $G$-equivariant dominant $k$-morphism $\sN_G^{\rm{var}} \to \sU_G^{\rm{var}}$ would necessarily restrict to a $G$-equivariant $k$-morphism $\sN_G^{\rm{reg}} \to \sU_G^{\rm{reg}}$, and similarly for $\sU_G^{\rm{var}} \to \sN_G^{\rm{reg}}$. We will prove the more precise statement that there are no $G$-equivariant $k$-morphisms $\sN_G^{\rm{reg}} \to \sU_G^{\rm{reg}}$ or $\sU_G^{\rm{reg}} \to \sN_G^{\rm{reg}}$.

    In general, if $H$ and $K$ are closed $k$-subgroup schemes of $G$, then there exists a $G$-equivariant $k$-morphism $G/H \to G/K$ if and only if there exists $g \in G(k)$ such that $gHg^{-1} \subset K$. Thus to prove the claim we need only show that if $u \in G(k)$ is a regular unipotent element and $X \in \fg(k)$ is a regular nilpotent element, then $Z_G(u) \not\subset Z_G(X)$ and $Z_G(X) \not\subset Z_G(u)$. For this, we may pass from $G$ to $\sD(G)$ to assume that $G$ is semisimple; let $\pi\colon \widetilde{G} \to G$ denote the universal cover.

    Suppose for the sake of contradiction that $Z_G(u) \subset Z_G(X)$ or $Z_G(X) \subset Z_G(u)$, where $u$ and $X$ are as above. For dimension reasons, this implies in particular that $Z_G(u)_{\rm{red}}^0 = Z_G(X)_{\rm{red}}^0$. Since $Z_G(u)$ and $Z_G(X)$ both contain $Z(G)$, it follows from Lemma~\ref{lemma:reg-unip-centralizer} (see also \cite[Lemma 3.4]{Integral-Springer}) that $Z_{\widetilde{G}}(\widetilde{u}) = Z_{\widetilde{G}}(\widetilde{X})$, where $\widetilde{u}$ and $\widetilde{X}$ are lifts of $u$ and $X$, respectively. By Lemma~\ref{lemma:centralizer-kernel-smooth} and Lemma~\ref{lemma:nilpotent-variety-stable}, the closed subgroup $N = \pi(Z_{\widetilde{G}}(\widetilde{u}))$ is therefore normal in both $Z_G(u)$ and $Z_G(X)$, and one obtains a monic homomorphism $Z_G(u)/N \to Z_G(X)/N$ or $Z_G(X)/N \to Z_G(u)/N$. By Lemma~\ref{lemma:centralizer-kernel-smooth}, the quotient $Z_G(u)/N$ is a nontrivial multiplicative type group scheme, while by Lemma~\ref{lemma:nilpotent-variety-stable}, the quotient $Z_G(X)/N$ is a nontrivial unipotent group scheme, so there are no monic homomorphisms $Z_G(u)/N \to Z_G(X)/N$ or $Z_G(X)/N \to Z_G(u)/N$ and we are done.
\end{proof}

\begin{example}\label{example:pglp-nonconstant}
    When $G = \PGL_p$, we can show the stronger claim that there is no \textit{non-constant} $G$-equivariant $k$-morphism $\sU_G^{\rm{var}} \to \sN_G^{\rm{var}}$ or $\sN_G^{\rm{var}} \to \sU_G^{\rm{var}}$. Indeed, by the proof of Proposition~\ref{prop:no-dominant-map} it is enough to show that if $u$ is regular unipotent (resp.\ $X$ is regular nilpotent), then for any non-regular nilpotent $Y \neq 0$ (resp.\ non-regular unipotent $v \neq 1$) we have $Z_G(u) \not\subset Z_G(Y)$ (resp.\ $Z_G(X) \not\subset Z_G(v)$). We will deal only with the the parenthetical claim, because the non-parenthetical statement involves using stability results for nilpotent elements which we have not pursued in this paper (though they are completely similar to those we'll use for unipotent elements).

    Let $\pi\colon \SL_p \to G$ be the quotient map, let $v$ be a non-regular unipotent element of $G(k)$, and lift $v$ to $v_0$ in $\SL_p$. By Lemma~\ref{lemma:twisted-centralizer-isogeny} and Proposition~\ref{prop:stability-type-a}, since $Z(\SL_p)$ has trivial intersection with the derived group of any centralizer of a semisimple element of $\SL_p(k)$, we have $Z_G(v) = \pi(Z_{\SL_p}(v_0))$. In particular, $Z_G(v)$ is contained in a proper parabolic subgroup of $G$, so it is enough to show that the same is not true of $Z_G(X)$ when $X \in \fg$ is regular nilpotent. By conjugacy, it is enough to show this when 
    \[
    X = \begin{pmatrix}
        0 & 1 & &  \\
        &0 &\ddots &  \\
        & &\ddots  &1 \\
        & &  &0
    \end{pmatrix}
    \]
    
    The only Borel $k$-subgroup of $G$ containing $X$ in its Lie algebra is the upper-triangular Borel, so it is enough to show that $Z_G(X)$ is not contained in any nontrivial group of block-upper-triangular matrices. For this, we compute directly that the element
    \[
    g = \begin{pmatrix}
        1 & & & & \\
        \eps &1 & & & \\
        &2\eps &\ddots & & \\
        & &\ddots &1 & \\
        & & &(p-1)\eps &1
    \end{pmatrix}
    \]
    lies in $Z_G(X)(k[\eps]/(\eps^2))$. Clearly $g$ does not lie in a nontrivial group of block-upper-triangular matrices, so neither does $Z_G(X)$.

    On the other hand, Theorem~\ref{theorem:stabilizer-of-Steinberg} shows that if $G = \SL_{p^2}/\mu_p$ then the map $\sU_{\SL_{p^2}} \to \sN_{\SL_{p^2}}$, $x \mapsto x^p - 1$ induces a nonconstant $G$-equivariant $k$-morphism $\sU_G \to \sN_G$.
\end{example}

\subsection{Ruling out an isomorphism}\label{subsection:no-springer-iso}

First, we recall the case of bad primes.

\begin{lemma}\label{lemma:no-springer-bad}
Let $k$ be a field of characteristic $p > 0$, and let $G$ be a connected reductive $k$-group such that $p$ is bad for $G$. Then there is no $G$-equivariant $k$-isomorphism $\sU_G^{\rm{var}} \cong \sN_G^{\rm{var}}$.
\end{lemma}

\begin{proof}
We may and do assume that $k$ is algebraically closed. Note that each of $\sU_G^{\rm{var}}$ and $\sN_G^{\rm{var}}$ contains an open dense $G$-orbit (the regular locus), so any $G$-equivariant isomorphism $\sU_G^{\rm{var}} \cong \sN_G^{\rm{var}}$ must send a regular unipotent $u$ element of $G(k)$ to a regular nilpotent element $X$ of $\fg$ satisfying $Z_G(u) = Z_G(X)$. By \cite[Thms.\ 4.12 and 5.9 b)]{Springer} and \cite[Theorem (b)]{Keny-nilpotent}, $Z_G(X)/Z(G)$ is connected, while $Z_G(u)/Z(G)$ is not. So such an isomorphism cannot exist.
\end{proof}

Next, we will rule out the case that $p$ divides $|\pi_1(\sD(G))|$ in a rather strong way: we will show that if $G$ is split and $p$ divides $|\pi_1(\sD(G))|$ then the unipotent and nilpotent varieties have different Picard groups, so they cannot be isomorphic as schemes.

\begin{lemma}\label{lemma:unip-pic}
Let $G$ be a split connected reductive group over a field $k$. Then $\Pic(\sU_G^{\rm{var}})$ contains $\Hom_{k\textrm{-}\mathrm{gp}}((\ker \pi_G)^0, \bGm)$, where $\pi_G\colon \widetilde{G} \to \sD G$ is the universal cover.
\end{lemma}

\begin{proof}
By Theorem~\ref{theorem:stabilizer-of-Steinberg}, the map $\sU_{\widetilde{G}} \to \sU$ is a $\ker(\pi_G)^0$-torsor, and since $k[\sU_{\widetilde{G}}]^\times = k^\times$ the result follows from Lemma~\ref{lemma:descent}.
\end{proof}

\begin{lemma}\label{lemma:picard-normal-cone}
Let $k$ be a field, and let $C \subset \bA^n$ be a normal cone (i.e., a normal closed subscheme which is preserved by $\bA^1$-scaling). Then $\Pic(C) = 0$.
\end{lemma}

\begin{proof}
If $\dim C \leq 1$ then the result is clear because either $C = \{0\}$ or $C \cong \bA^1$, so assume $\dim C \geq 2$ Let $\sL$ be a line bundle on $C$, and let $\pi: X \to C$ denote the blowup of $C$ at the cone point $0$. Let $P = (C - \{0\})/\bGm$, so $P$ is a projective variety. Note that $C$ is the affine cone of $P \subset \bP^{n-1}$, so by \cite[II, Remarque 8.7.8]{EGA} there is a map $X \to P$ realizing $X$ as the total space of the vector bundle $\sO_P(1)$. The pullback map $\Pic(P) \to \Pic(X)$ is surjective by \cite[Proposition 1.9]{Fulton}, and it is injective because it has a section coming from the exceptional divisor of $\pi$. Note that $(\pi^*\sL)|_{\pi^{-1}(0)} = \pi^*(\sL|_{0})$ is trivial, so this shows that $\pi^*\sL$ is also trivial. Thus $\sL$ is trivial in $\Pic(C - \{0\})$. But the restriction map $\Pic(C) \to \Pic(C - \{0\})$ is injective by normality of $X$ (since $\Pic(C) \to \Cl(C)$ is injective and $0$ is of codimension $\geq 2$), so indeed $\sL$ is trivial.
\end{proof}

\begin{theorem}\label{theorem:no-springer-iso}
Let $G$ be a split connected reductive group over a field $k$ of characteristic $p > 0$. If $p$ divides $|\pi_1(\sD(G))|$, then $\sU_G^{\rm{var}}$ and $\sN_G^{\rm{var}}$ are not isomorphic as $k$-schemes.
\end{theorem}

\begin{proof}
Let $\sU = \sU_G^{\rm{var}}$ and $\sN = \sN_G^{\rm{var}}$. By Corollary~\ref{cor:unip-small-char}, $\sU$ is normal, so if $\sN$ is \textit{not} normal, it is certainly not isomorphic to $\sU$. Now if $\sN$ \textit{is} normal, then Lemma~\ref{lemma:picard-normal-cone} shows that $\Pic(\sN) = 0$. On the other hand, since $(\ker \pi_G)^0 \neq 0$ by assumption, Lemma~\ref{lemma:unip-pic} shows that $\Pic(\sU) \neq 0$.
\end{proof}

\begin{remark}\label{remark:guess-at-full-pic}
	It seems likely that $\Pic(\sU_G^{\rm{var}}) \cong \Hom_{k\textrm{-}\mathrm{gp}}((\ker \pi_G)^0, \bGm)$ in general, and an earlier version of this paper made this claim. However, the calculation relied on the claim that in Lemma~\ref{lemma:descent}, if one takes $H$ to be diagonalizable, then the map $i$ of that lemma is an isomorphism. This is erroneous, as Remark~\ref{remark:alpha2-torsor} shows.
\end{remark}

\section{Application to the moduli space of L-parameters}\label{section:moduli}

In this section, we prove Theorem~\ref{theorem:intro-l-param}. We begin with some generalities.

\subsection{Some generalities}

Let $S$ be a scheme, let $G$ be a finitely presented $S$-group scheme, let $X$ be a finitely presented $S$-scheme equipped with an action of $G$, and let $\sigma\colon X \to X$ be a $G$-equivariant $S$-morphism. Define
\[
\sA_{G, X, \sigma}\coloneqq \{(g, x) \in G \times_S X\colon\, g \cdot x = \sigma(x)\},
\]
a finitely presented $S$-scheme. We will write $\sA_{G, X}$ in place of $\sA_{G, X, \id_X}$ for simplicity.

\begin{example}
    Schemes of the form $\sA_{G, X, \sigma}$ come up in at least the following two settings.
    \begin{enumerate}
        \item If $G$ is a linear algebraic group over a field $k$, and $G$ acts on itself by conjugation, then $\sA_{G, G}$ is the \textit{commuting scheme} of $G$.
        \item Let $q > 1$ be a prime power, let $G$ be a reductive group scheme over a base scheme $S$ on which $q$ is invertible, and let $G$ act on itself by conjugation. If $[q](g) = g^q$, then $\sA_{G, G, [q]}$ is the moduli space of tame representations of the ``discretized Weil group" of a local field $F$ with residue cardinality $q$ (as in \cite{DHKM}).\footnote{Of course, this can also be jazzed up to include spaces of L-parameters, and we will do this below.}
    \end{enumerate}
\end{example}

In the following subsection, we will take $X = G$ and $\sigma(g) = z \cdot g^q$ for some $z \in Z(G)(S)$. However, for the sake of future reference, we prefer to work in the above generality, and we will prove slightly more than we need. The citations in the lemma statements below record that each statement generalizes a recent result concerning moduli spaces of L-parameters, at least once one restricts attention to tame L-parameters. (The results of \cite[\S 4.1]{DHKM} show that one can often reduce to this case.)

\begin{lemma}\label{lemma:irr-comps-of-l-param}
    Suppose $S = \Spec k$ for an algebraically closed field $k$. The $k$-scheme $\sA_{G, X, \sigma}$ is equidimensional of dimension $\dim(G)$ if and only if there are only finitely many $\sigma$-stable $G$-orbits in $X$. If this is the case, then
    \begin{enumerate}
        \item \cite[Corollary 2.4]{DHKM}, \cite[Corollary 2.7]{Shotton-irreducible} the irreducible components of $\sA_{G, X, \sigma}$ are parameterized by $G^0$-orbits of pairs $(x, C)$, where $x \in X(k)$ and $C$ is a connected component of $\Transp_G(x, \sigma(x))$.
        \item \cite[Remark 6.23]{Booher-Tang} every point $(g, x) \in \sA_{G, X, \sigma}(k)$ such that $\Stab_G(x)_{\rm{red}}$ is normal in $\Stab_G(x)$ lies in an irreducible component $D$ whose underlying reduced subscheme is smooth at $(g, x)$.
    \end{enumerate}
\end{lemma}

\begin{proof}
    Let $\pi\colon \sA_{G, X, \sigma} \to X$ be the projection map $(g, x) \mapsto x$, so the image of $\pi(k)$ is the set of elements of $X(k)$ with $\sigma$-stable $G(k)$-orbit. Let $Y$ be a (smooth) $\sigma$-stable locally closed $G$-orbit in $X$. The map $\pi^{-1}(Y) \to Y$ is visibly flat, being $G$-equivariant with respect to the action $g \cdot (h, y) = (ghg^{-1}, g \cdot y)$ on the left side, so because the fiber $\pi$ over a point $x \in Y(k)$ is isomorphic to $\Stab_G(x)$, we see that $\pi^{-1}(Y)$ is equidimensional and $\dim \pi^{-1}(Y) = \dim Y + \dim \Stab_G(x) = \dim G$. Moreover, if $Y' \neq Y$ is another $\sigma$-stable locally closed $G$-orbit in $X$, then $\pi^{-1}(Y) \cap \pi^{-1}(Y') = \emptyset$. Since $\sA_{G, X, \sigma}$ is of finite type over $k$, if it is of dimension $\dim G$ then there can only be finitely many such $Y$, i.e., there are only finitely many $\sigma$-stable $G$-orbits in $X$. Conversely, if $Y_1, \dots, Y_n$ are the finitely many $\sigma$-stable $G$-orbits in $X$, then $\sA_{G, X, \sigma} = \bigcup_{i=1}^n \pi^{-1}(Y_i)$, so $\sA_{G, X, \sigma}$ is equidimensional of dimension $\dim G$.

    Now suppose $\sA_{G,X,\sigma}$ is of pure dimension $\dim G$. For dimension reasons, the above paragraph shows that the irreducible components of $\sA_{G, X, \sigma}$ are parameterized by a $\sigma$-stable $G$-orbit $Y \subset X$ and an irreducible component of $\pi^{-1}(Y)$. If $x \in Y(k)$ and $C$ is a component of $\Transp_G(x, \sigma(x))$, then there is a natural map $\alpha\colon G^0_{\rm{red}} \times C \to \pi^{-1}(Y)$ given by $(g, h) \mapsto (ghg^{-1}, g \cdot x)$. The set-theoretic image $\im\alpha$ of $\alpha$ is constructible by Chevalley's theorem, so it contains a dense locally closed subscheme $D$. Since $G^0_{\rm{red}} \times C$ is irreducible, it follows that $D$ is irreducible, and it is clearly of dimension $\dim G$. Thus $D$ is dense in an irreducible component of $\pi^{-1}(Y)$, and it is easy to see that every irreducible component of $\pi^{-1}(Y)$ can be obtained in this way. This shows (1).
    
    For (2), note that if $Y$ is the $\sigma$-stable $G$-orbit in $X$ containing $x$, then $\pi^{-1}(Y)$ is smooth: indeed, the map $\pi^{-1}(Y) \to Y$ is flat with fibers isomorphic to the smooth $k$-scheme $\Stab_G(x)$. The closure of $\pi^{-1}(Y)$ is a union of (underlying reduced subschemes of) irreducible components of $\sA_{G, X, \sigma}$, and $\pi^{-1}(Y)$ is open in $\ov{\pi^{-1}(Y)}$. Thus $(g, x)$ is a smooth point in some reduced irreducible component of $\sA_{G, X, \sigma}$, and (2) follows.
\end{proof}

\begin{lemma}\label{lemma:flatness-of-l-param}
    \cite[Corollary 2.5]{DHKM}, \cite[Proposition 3.1.6]{Zhu} Suppose $S$ is irreducible, $G$ is $S$-flat, $X$ is $S$-smooth, and for each algebraically closed field $k$ over $S$, there are only finitely many $\sigma$-stable $G_k$-orbits in $X_k$. Then $\sA_{G, X, \sigma}$ is $S$-flat with lci fibers.
\end{lemma}

\begin{proof}
    Note that there is a Cartesian diagram
    \[
    \begin{tikzcd}
        \sA_{G, X, \sigma} \arrow[r, "{(g, x)} \mapsto x"] \arrow[d]
            &X \arrow[d, "\Delta_X"] \\
        G \times_S X \arrow[r]
            &X \times_S X
    \end{tikzcd}
    \]
    where the bottom horizontal map is given by $(g, x) \mapsto (\sigma(x), g \cdot x)$. Because $X$ is $S$-smooth, the map $\Delta_X$ is a locally closed embedding which is locally defined by the vanishing of $\dim X$ functions. Thus $\sA_{G, X, \sigma}$ is a locally closed subscheme of $G \times_S X$ which is locally defined by the vanishing of $\dim X$ functions, so it is enough to show that for each algebraically closed field $k$ over $S$, the fiber $(\sA_{G, X, \sigma})_k$ is of dimension $\leq \dim G_k$. This is proven in Lemma~\ref{lemma:irr-comps-of-l-param}.
\end{proof}

Given a $\sigma$-stable $G$-orbit $\cO \subset X$, let $\sA_{G,X,\sigma,\cO}$ denote the locally closed subscheme $\pi^{-1}(\cO)$ of $\sA_{G,X,\sigma}$, where $\pi\colon \sA_{G,X,\sigma} \to X$ is the projection map. As in the proof of Lemma~\ref{lemma:irr-comps-of-l-param}, if there are only finitely many $\sigma$-stable $G$-orbits in $X$ then each connected component of $\pi^{-1}(\cO)$ is (topologically) open in an irreducible component of $\sA_{G,X,\sigma}$.

\subsection{L-parameters}\label{ss:l-param}

Let $F$ be a local field of residue cardinality $q$, let $G$ be a reductive group scheme over a base scheme $S$ in which $q$ is invertible, and suppose that $G$ is equipped with a finite tame action of the Weil group $W_F$ preserving a Borel pair $(B, T)$. Let $\rm{Fr}$ be a lift of Frobenius, and let $\sigma$ be a lift of a generator of tame inertia. Let $\mu$ be a finite flat closed $S$-subgroup scheme of $Z(G)^\sigma$, and define the $S$-scheme $\sX_{G, \mu}$ by
\[
\sX_{G, \mu} = \{(\Phi, \Sigma) \in G/\mu \times G\colon\, \Phi\cdot \prescript{\rm{Fr}}{}{\Sigma} \cdot \prescript{\sigma^q}{}{\Phi^{-1}} \equiv \prod_{i=0}^{q-1} \prescript{\sigma^i}{}{\Sigma} \pmod{\mu}\}.
\]
If $z \in \mu(S)$, then define
\[
\sX_{G, \mu, z} = \{(\Phi, \Sigma) \in G/\mu \times G\colon\, z\cdot\Phi\cdot \prescript{\rm{Fr}}{}{\Sigma} \cdot \prescript{\sigma^q}{}{\Phi^{-1}} = \prod_{i=0}^{q-1} \prescript{\sigma^i}{}{\Sigma}\}.
\]
In the notation of the previous section, we have $\sX_{G, \mu, z} = \sA_{G/\mu, G, \tau}$, where $G/\mu$ acts on $G$ by $g \cdot h = g \cdot \prescript{\rm{Fr}}{}{h} \cdot \prescript{\sigma^q}{}{h}^{-1}$ and $\tau(g) = z^{-1} \cdot \prod_{i=0}^{q-1} \prescript{\sigma^i}{}{g}$. There are only finitely many $\tau$-stable $G/\mu$-orbits in $G$, so Lemma~\ref{lemma:irr-comps-of-l-param} and Lemma~\ref{lemma:flatness-of-l-param} show that $\sX_{G, \mu, z}$ is $S$-flat and lci of the same relative dimension as $G$.

Let $\sX_G = \sX_{G, 1, 1}$, so $\sX_G$ is the moduli space of L-parameters valued in $\prescript{L}{}{G} = G \rtimes W_0$, where $W_0$ is a chosen finite quotient of $W_F$ through which the action factors. We view $\sX_{G, \mu, z}$ as a ``twisted version" of this moduli space and we will relate it to $\sX_{G/\mu}$ using Lemma~\ref{lemma:multiplicity-behavior}. Note that there is a Cartesian diagram
\[
\begin{tikzcd}
    \sX_{G, \mu} \arrow[d] \arrow[r]
        &\sX_{G/\mu} \arrow[d, "{(\Phi, \Sigma)} \mapsto \Sigma"] \\
    G \arrow[r]
        &G/\mu
\end{tikzcd}
\]
so $\sX_{G,\mu} \to \sX_{G/\mu}$ is a $\mu$-torsor and thus $\sX_{G, \mu}$ is also $S$-flat and lci of the same relative dimension as $G$. If $S = \Spec K$ for a field $K$ of characteristic $0$, then \cite[Proposition 2.8]{DHKM} implies that $\sX_{G,\mu}$ is reduced since this is true of $\mu$ and $\sX_{G/\mu}$. For such an $S$, each scheme $\sX_{G,\mu,z}$ is an open subscheme of $\sX_{G,\mu}$, so $\sX_{G,\mu,z}$ is also reduced.



\begin{lemma}\label{lemma:finite-gen-iso}
    If $S = \Spec A$ for a domain $A$ with fraction field $K$ such that $\chara K = 0$ and $\mu(K) = \mu(\ov{K})$, then $\coprod_{z \in \mu(A)} \sX_{G, \mu, z} \to \sX_{G,\mu}$ is a finite surjective map whose $K$-fiber is an isomorphism.
\end{lemma}

\begin{proof}
    This is clear from the definitions.
\end{proof}

If $S = \Spec k$ for a field $k$, then for $g \in G(k)$ let $\sX_{G,\mu,z,g}$ denote the closed subscheme of $\sX_{G,\mu,z}$ consisting of those $(\Phi, \Sigma)$ such that $\Sigma$ lies in the $\sigma$-conjugacy class of $g$. Note that $\sX_{G,\mu,g}$ is of dimension $\dim G$ (as one can deduce from properties of the map $\sX_{G,\mu,g} \to G$, $(\Phi, \Sigma) \mapsto \Sigma$), so $(sX_{G,\mu,g})_\rm{red}$ is a union of connected components of $(\sX_{G,\mu})_{\rm{red}}$.

\begin{lemma}\label{lemma:unip-comp-mult}
    Suppose $S = \Spec k$ for a perfect field $k$, and let $g \in G(k)$ be a $\sigma$-regular element such that $Z_{G,\sigma^q}(\prescript{\rm{Fr}}{}{g})/Z(G)^\sigma$ is smooth. Then each irreducible component of the scheme
    $\sX_{G,\mu,z,g}$
    is of generic multiplicity $|Z_{G,\sigma^q}(\prescript{\rm{Fr}}{}{g})/(\mu \cdot Z_{G,\sigma^q}(\prescript{\rm{Fr}}{}{g})_{\rm{red}})|$.
\end{lemma}

\begin{proof}
    We may and do assume $k = \ov{k}$. Let $\sX = \sX_{G,\mu,z,g}$, and let $\sX' = \sX_{\rm{red}}$.
    If $(\Phi_0, g) \in \sX(k)$, then the conjugation map $\alpha\colon G \times ((\Phi_0 \cdot (Z_{G,\sigma^q}(\prescript{\rm{Fr}}{}{g})/\mu)_{\rm{red}}) \times \{g\}) \to \sX$ factors through a surjective map to $\sX'$; let $\sX'_0$ be the largest open subscheme of $\sX'$ over which $\alpha$ is flat, and note that $\sX'_0$ is dense in $\sX'$ because $\sX'$ is reduced and $k$ is perfect. Note that $\sX'_0$ is stable under $G$-conjugation, so $\alpha^{-1}(\sX'_0) = G \times C$ for some open subscheme $C$ of $(\Phi_0 \cdot (Z_{G,\sigma^q}(\prescript{\rm{Fr}}{}{g})/\mu)_{\rm{red}}) \times \{g\}$.
    
    Fix an irreducible component $\sX'_1$ of $\sX'_0$, let $\sX_1$ be the open subscheme of $\sX$ with the same underlying space as $\sX'_1$, and let $m\colon (Z(G)^\sigma/\mu)^0 \times \sX'_1 \to \sX_1$ be the map $(z_0, (\Phi, \Sigma)) \mapsto (z_0\Phi, \Sigma)$. We claim that $m$ is faithfully flat. Indeed, note that the projection maps $\sX_1 \to C_{g,\rm{Fr}}$ and $(Z(G)^\sigma/\mu)^0 \times \sX'_1 \to C_{g,\rm{Fr}}$ to $\Sigma$ are both flat, as one sees by pulling back via the faithfully flat map $G \to C_{g,\rm{Fr}}$, $h \mapsto hg\sigma(h)^{-1}$. By the fibral flatness criterion, it is therefore enough to show that $m$ is faithfully flat on fibers over $C_{g,\rm{Fr}}$. By homogeneity, it is enough to show this on the fiber over $g$, where the map is simply the multiplication map $(Z(G)^\sigma/\mu)^0 \times (\Phi_0 \cdot (Z_{G,\sigma^q}(\prescript{\rm{Fr}}{}{g})/\mu)_{\rm{red}}) \to \Phi_0 \cdot Z_{G,\sigma^q}(\prescript{\rm{Fr}}{}{g})/\mu$. Since we have assumed that $Z_{G,\sigma^q}(\prescript{\rm{Fr}}{}{g})/Z(G)^\sigma$ is smooth, this map is an epimorphism of fppf sheaves, and hence it is faithfully flat by \cite[Expos\'e VI\textsubscript{B}, Remarque 9.2.2]{SGA3I}. This proves the claim.
    
    Next we claim that 
    \[
    m^{-1}(\sX'_1) = ((Z(G)^\sigma/\mu)^0 \cap (Z_{G,\sigma^q}(\prescript{\rm{Fr}}{}{g})/\mu)_{\rm{red}}) \times \sX'_1.
    \]
    To see this, note that the right side is evidently contained in the left, and to check equality we may pull back by a faithfully flat map, e.g., $\alpha^{-1}(\sX'_1) \to \sX'_1$. Note that $\alpha^{-1}(\sX'_1) = G \times C_1$ for some clopen subscheme $C_1$ of $C$. The claim is now clear from the fact that the map $m^{-1}(\sX'_1) \times_{\sX'_1} \alpha^{-1}(\sX'_1) \to \alpha^{-1}(\sX'_1)$ is given by $(z_0, g, c_1) \mapsto (z_0g, c_1)$.
    
    By Lemma~\ref{lemma:mult-flat-cover} applied to $m$, since $\mu(\sX'_1) = 1$ we have
    \[
    \mu((Z(G)^\sigma/\mu)^0) = \mu(\sX_0) \cdot \mu((Z(G)^\sigma/\mu)^0 \cap (Z_{G,\sigma^q}(\prescript{\rm{Fr}}{}{g})/\mu)_{\rm{red}}).
    \]
    This shows
    \begin{align*}
    \mu(\sX_0)
    &= \mu((Z(G)^\sigma/\mu)^0 \cap (Z_{G,\sigma^q}(\prescript{\rm{Fr}}{}{g})/\mu)_{\rm{red}})^{-1} \cdot \mu((Z(G)^\sigma/\mu)^0) \\
    &= |Z_{G,\sigma^q}(\prescript{\rm{Fr}}{}{g})/(\mu \cdot Z_{G,\sigma^q}(\prescript{\rm{Fr}}{}{g})_{\rm{red}})|,
    \end{align*}
    as desired.
\end{proof}

Form now on, assume that $S = \Spec A$, where $A$ is a complete discrete valuation ring with uniformizer $\pi$, algebraically closed residue field $k = A/\pi$ of characteristic $\ell > 0$, and fraction field $K$ of characteristic $0$. Let $\sX_{G, \mu}^{\Sigma\textrm{-}\mathrm{reg}}$ denote the open subscheme of $\sX_{G, \mu}$ consisting of those pairs $(\Phi, \Sigma)$ such that $\Sigma$ has $\sigma$-regular fibers over $S$. Define $\sX_{G, \mu, z}^{\Sigma\textrm{-}\mathrm{reg}}$ similarly.

\begin{lemma}\label{lemma:asc-mult}
	Suppose the pair $(G, \sigma)$ is almost simply connected, let $\mu \subset Z(G)^\sigma$ be a finite flat $A$-subgroup scheme, and let $C$ be an irreducible component of $(\sX_{G,\mu})^{\Sigma\textrm{-}\mathrm{reg}}_K$ which is geometrically irreducible. If $D$ is an irreducible component of $\ov{C} \cap (\sX_{G,\mu})^{\Sigma\textrm{-}\mathrm{reg}}_k$, and let $g \in G(A)$ be a section with $\sigma$-regular fibers such that the map $C \to G$, $(\Phi, \Sigma) \mapsto \Sigma$ factors through $\chi_\sigma^{-1}(\chi_\sigma(g))$.  Then $D$ is of generic multiplicity
	\[
	\mu(D) = \frac{|Z_{G_k,\sigma^q}(\prescript{\rm{Fr}}{}{g_k})/(\mu_k \cdot Z_{G_k,\sigma^q}(\prescript{\rm{Fr}}{}{g_k})_{\rm{red}})|}{|(Z_{G_K,\sigma^q}(\prescript{\rm{Fr}}{}{g_K})/(\mu_K \cdot Z_{G_K,\sigma^q}(\prescript{\rm{Fr}}{}{g_K})^0))[\ell^\infty]|},
	\]
	where $[\ell^\infty]$ refers to the subgroup of $\ell$-power torsion.
\end{lemma}

\begin{proof}
	Note that $C$ lies in $(\sX_{G,\mu,z})_K$ for some $z \in \mu(A)$ by Lemma~\ref{lemma:finite-gen-iso}, and $\ov{C}$ is then a closed subscheme of $\sX_{G,\mu,z}$. Let $\sC$ be the open subscheme of $\ov{C} \cap \sX_{G,\mu,z}^{\Sigma\textrm{-reg}}$ obtained by removing all components other than $D$ from the special fiber, and let $(h, g) \in \sC(A)$ be a section. By Lemma~\ref{lemma:irr-comps-of-l-param}, the irreducible components of $(\sX_{G,\mu,z})_{\ov{K}}$ containing a point with second coordinate $g$ are all of the form $C_z = \{(z\Phi, \Sigma)\colon (\Phi, \Sigma) \in C_{\ov{K}}\}$ for some $z \in (Z(G)^\sigma/\mu)(\ov{K})$, and we have $C_z = C_{z'}$ if and only if $z$ and $z'$ lie in the same component of $Z_{G_{\ov{K}}, \sigma^q}(\prescript{\rm{Fr}}{}{g_{\ov{K}}})/\mu_{\ov{K}}$. Since $Z(G)^\sigma$ is of multiplicative type, it can be decomposed as $Z(G)^\sigma \cong T \times_A M$, where $T$ is an $A$-torus and $M$ is a finite flat $A$-group scheme of multiplicative type. Since $T$ has connected fibers, in the above discussion we may always take $z$ to lie in the image of $M(K) = M(A)$, at least after first passing to a finite extension of $K$. Note that by the same discussion, an element $z \in M(A)$ preserves $D$ under multiplication on the first coordinate if and only if its image in $G(k)$ lies in $\mu(k) \cdot Z_{G_k, \sigma^q}(\prescript{\rm{Fr}}{}{g_k})^0(k)$.
	
	Let $H$ be the schematic closure of $\mu_K \cdot Z_{G_K,\sigma^q}(\prescript{\rm{Fr}}{}{g_K})^0$ in $G$, so $H$ is a normal closed $A$-subgroup scheme of $Z_{G,\sigma^q}(\prescript{\rm{Fr}}{}{g})$ such that $Z_{G,\sigma^q}(\prescript{\rm{Fr}}{}{g})/H$ is finite flat of multiplicative type. (Flatness comes from Lemma~\ref{lemma:flatness-of-fixed-point-scheme} and Corollary~\ref{cor:asc-smooth}.) By \cite[Theorem 8.1]{SteinbergEndomorphisms} and Corollary~\ref{cor:asc-smooth}, if $(g_K, \sigma) = t_Ku_K$ is the Jordan decomposition in the (disconnected) reductive $K$-group $G \rtimes \langle \sigma \rangle$, then $Z_{G_K}(t_K)$ is connected. Moreover, by Lemma~\ref{lemma:reg-unip-centralizer} and the fact that $\chara K = 0$ (so all $K$-group schemes are smooth), we have $Z_{G_K,\sigma}(g) = Z_{Z_{G_K}(t_K)}(u_K) = Z(Z_{G_K}(t_K)) \cdot Z_U(u)$, where $U$ is the unipotent radical of the unique Borel $K$-subgroup of $G_K$ containing $u$. Since $M$ intersects $Z_U(u)$ trivially, we find that $z \in M(A)$ has image in $G(k)$ which lies in $\mu(k) \cdot Z_{G_k, \sigma^q}(\prescript{\rm{Fr}}{}{g_k})^0(k)$ if and only if the image of $z$ in $(Z_{G_K,\sigma^q}(\prescript{\rm{Fr}}{}{g_K})/\mu_K \cdot Z_{G_K,\sigma^q}(\prescript{\rm{Fr}}{}{g_K})^0)(K)$ is $\ell$-power torsion.
	
	Now let $z_1, \dots, z_m \in M(A)$ be representatives for the $\ell$-power torsion part of the image of $M(A)$ in $(Z_{G_K,\sigma^q}(\prescript{\rm{Fr}}{}{g_K})/\mu_K \cdot Z_{G_K,\sigma^q}(\prescript{\rm{Fr}}{}{g_K})^0)(K)$. By the above discussions, the components $C_{z_1}, \dots, C_{z_m}$ exhaust the irreducible components of $(\sX_{G,\mu,z})_K$ which lift $D$, so Lemma~\ref{lemma:unip-comp-mult} and Lemma~\ref{lemma:multiplicity-behavior} combine to give the result.
\end{proof}

Recall the notation $\pi_{G,\sigma}\colon \widetilde{G}_\sigma \to G$ from Lemma~\ref{lemma:almost-simply-connected}, and let $\mu = \ker \pi_{G,\sigma}$. Let $f\colon \sX_{\widetilde{G}_\sigma, \mu} \to \sX_G$ be the natural map.

\begin{theorem}\label{theorem:l-param-mult}
    Suppose that $K$ is large enough that every irreducible component of $(\sX_G)_K$ is geometrically irreducible and $\mu(K) = \mu(\ov{K})$. Let $C$ be an irreducible component of $(\sX_G)_K$, and let $D$ be a component of $(\ov{C} \cap \sX_G^{\Sigma\textrm{-}\rm{reg}})_k$. Let $g \in G(A)$ be a section with $\sigma$-regular fibers such that the map $C \to G$, $(\Phi, \Sigma) \mapsto \Sigma$ factors through $\chi_\sigma^{-1}(\chi_\sigma(g))$. Let $\widetilde{g} \in \widetilde{G}_\sigma(A)$ lift $g$. Then $D$ is of generic multiplicity
        \[
        \mu(D) = \frac{|\Stab_{\mu_k}(C_{\widetilde{g}_k,\sigma})| \cdot |Z_{(\widetilde{G}_\sigma)_k,\sigma}(\widetilde{g}_k)/(\mu \cdot Z_{(\widetilde{G}_\sigma)_k,\sigma}(\widetilde{g}_k)_{\rm{red}})|}{|\Stab_{\mu_K}(C_{\widetilde{g}_K,\sigma})| \cdot |(Z_{(\widetilde{G}_\sigma)_K,\sigma}(g_K)/(\mu_K \cdot Z_{(\widetilde{G}_\sigma)_K,\sigma}(g_K)^0))[\ell^\infty]|}.
        \]
\end{theorem}

Of course, using Theorem~\ref{theorem:stabilizer-of-Steinberg}, Theorem~\ref{theorem:torsor-identification}, and \cite[Theorem 9.1(a)]{SteinbergEndomorphisms}, one can make this expression somewhat more explicit; we have partially done this in the introduction.

\begin{proof}
    By Lemma~\ref{lemma:finite-gen-iso}, the map $f\colon \sX_{\widetilde{G}_\sigma, \mu} \to \sX_G$ is a $\mu$-torsor. Let $\sC^0$ be the open subscheme of $\ov{C} \cap \sX_G^{\Sigma\textrm{-reg}}$ obtained by removing all components other than $D$ from the special fiber. Note that $\mu(A)$ acts transitively both on the set of connected components of $f^{-1}(\sC^0)_K$ and the set of connected components of $f^{-1}(\sC^0)_k$, and the stabilizers are isomorphic to $\Stab_{\mu(A)}(C_{\widetilde{g}_K,\sigma})$ and $\Stab_{\mu(A)}(C_{\widetilde{g}_k,\sigma})$, respectively. We have
    \[
    \Stab_{\mu(A)}(C_{\widetilde{g}_k,\sigma}) \cong \mu^0(A) \times \Stab_{\mu(k)}(C_{\widetilde{g}_k,\sigma})
    \]
    since any element of $\mu^0(A)$ reduces to the identity element in $\mu(k)$. Thus if $\sC$ is a connected component of $f^{-1}(\sC^0)$ then Lemma~\ref{lemma:multiplicity-behavior} shows
    \begin{equation}\label{eqn:mu-1}
    \mu(D) = \frac{1}{|\mu^0(A)| \cdot |\Stab_{\mu(k)}(C_{\widetilde{g}_k,\sigma})|}\mu(\sC_k) \cdot \deg((\sC_k)_{\rm{red}} \to D_{\rm{red}}).
    \end{equation}
    Since the map $C_{\widetilde{g}_k,\sigma} \to C_{g_k,\sigma}$ is a $\Stab_{\mu_k}(C_{\widetilde{g}_k,\sigma})$-torsor, we have
    \begin{equation}\label{eqn:mu-2}
    \deg((\sC_k)_{\rm{red}} \to D_{\rm{red}}) = |\Stab_{\mu_k}(C_{\widetilde{g}_k,\sigma})|,
    \end{equation}
    so we only need to compute $\mu(\sC_k)$.

    Let $D_1, \dots, D_m$ be the connected components of $\sC_K$, and let $\sD_1, \dots, \sD_m$ be the closures of $D_1, \dots, D_m$ in $\sC$, so the map $\coprod_{i=1}^m \sD_i \to \sC$ is a finite morphism whose $K$-fiber is an isomorphism. Moreover, for each $i$ the map $((\sD_i)_k)_{\rm{red}} \to (\sC_k)_{\rm{red}}$ is an isomorphism, so Lemma~\ref{lemma:multiplicity-behavior} again implies
    \[
    \mu(\sC_k) = \sum_{i=1}^m \mu((\sD_i)_k).
    \]
    Note that each $\sD_i$ lies in $\sX_{G, \mu, z}$ for some $z \in \mu(A)$, so by Corollary~\ref{cor:asc-smooth}, Lemma~\ref{lemma:asc-mult}, and the fact that $\mu$ is $\rm{Fr}$-stable, we have
    \begin{align*}
    |\mu((\sD_i)_k)| &= \frac{|Z_{(\widetilde{G}_\sigma)_k,\sigma^q}(\prescript{\rm{Fr}}{}{\widetilde{g}_k})/(\mu_k \cdot Z_{(\widetilde{G}_\sigma)_k,\sigma^q}(\prescript{\rm{Fr}}{}{\widetilde{g}_k})_{\rm{red}})|}{|(Z_{(\widetilde{G}_\sigma)_K,\sigma^q}(\prescript{\rm{Fr}}{}{\widetilde{g}_K})/(\mu_K \cdot Z_{(\widetilde{G}_\sigma)_K,\sigma^q}(\prescript{\rm{Fr}}{}{\widetilde{g}_K})^0))[\ell^\infty]|} \\
    &= \frac{|Z_{(\widetilde{G}_\sigma)_k,\sigma}(\widetilde{g}_k)/(\mu_k \cdot Z_{(\widetilde{G}_\sigma)_k,\sigma}(\widetilde{g}_k)_{\rm{red}})|}{|(Z_{(\widetilde{G}_\sigma)_K,\sigma}(\widetilde{g}_K)/(\mu_K \cdot Z_{(\widetilde{G}_\sigma)_K,\sigma}(\widetilde{g}_K)^0))[\ell^\infty]|}
    \end{align*}
    Moreover, by the description of the $\mu(A)$-action on the sets of connected components of $f^{-1}(\sC^0)_k$ and $f^{-1}(\sC^0)_K$ above, we have
    \[
    m = \frac{|\mu^0(A)| \cdot |\Stab_{\mu(k)}(C_{\widetilde{g}_k,\sigma})|}{|\Stab_{\mu_K}(C_{\widetilde{g}_K,\sigma})|}.
    \]
    In sum, we find
    \begin{equation}\label{eqn:mu-3}
    \mu(\sC_k) = \frac{|\mu^0(A)| \cdot |\Stab_{\mu(k)}(C_{\widetilde{g}_k,\sigma})| \cdot |Z_{(\widetilde{G}_\sigma)_k,\sigma}(\widetilde{g}_k)/(\mu_k \cdot Z_{(\widetilde{G}_\sigma)_k,\sigma}(\widetilde{g}_k)_{\rm{red}})|}{|\Stab_{\mu_K}(C_{\widetilde{g}_K,\sigma})| \cdot |(Z_{(\widetilde{G}_\sigma)_K,\sigma}(\widetilde{g}_K)/(\mu_K \cdot Z_{(\widetilde{G}_\sigma)_K,\sigma}(\widetilde{g}_K)^0))[\ell^\infty]|}.
    \end{equation}
    Combining (\ref{eqn:mu-1}), (\ref{eqn:mu-2}), and (\ref{eqn:mu-3}) gives the result.
    \end{proof}

We conclude by generalizing \cite[Theorem 3.7]{Shotton-irreducible} for future reference in a way partially suggested by \cite[Remark 3.9]{Shotton-irreducible}. Note that $\rm{Fr}\colon G \to G$ induces an isomorphism $G/\!/_\sigma G \to G/\!/_{\sigma^q} G$, and the $q$-power map $G \times \{\sigma\} \to G \times \{\sigma^q\}$, $(g, \sigma) \mapsto (g, \sigma)^q$ induces a map $[q]\colon G/\!/_\sigma G \to G/\!/_{\sigma^q} G$. Thus we obtain a morphism $\rm{Fr}^{-1}[q]\colon G/\!/_\sigma G \to G/\!/_\sigma G$. Note that the action of $\mu$ on $G$ by translation commutes with twisted conjugation, so there is a natural map $\alpha\colon \mu \times_S G/\!/_\sigma G \to G/\!/_\sigma G \times_S G/\!/_\sigma G$ given by $\alpha(z, x) = (z\cdot x, x)$. Define the $S$-scheme $B_{G, \mu}$ by the Cartesian diagram
\[
\begin{tikzcd}
    B_{G, \mu} \arrow[d] \arrow[r]
        &G/\!/_\sigma G \arrow[d, "\id \times \mathrm{Fr}^{-1}{[q]}"] \\
    \mu \times_S G/\!/_\sigma G \arrow[r, "\alpha"]
        &G/\!/_\sigma G \times_S G/\!/_\sigma G
\end{tikzcd}
\]
For $z \in \mu(S)$, define $B_{G, \mu, z}$ to be the preimage of $\{z\} \times_S G/\!/_\sigma G$ under the map $B_{G, \mu} \to \mu \times_S G/\!/_\sigma G$.

\begin{remark}
    Note that $B_{G, \mu, z}$ is a closed subscheme of $G/\!/_\sigma G$, but $B_{G, \mu}$ is not in general: indeed, Lemma~\ref{lemma:Springer-iso-6.4} shows that if $S = \Spec k$ for a field $k$ and $W_F$ acts trivially on $G$, then the map $\alpha|_{(\mu^0 \cap Z(\sD G)) \times_k \{\chi_\sigma(1)\}}$ is constant. It would be interesting to give descriptions of $B_{G,\mu}$ and $B_{G,\mu,z}$ akin to those in \cite{Li-GG}, \cite{Li-Shotton}.
\end{remark}

\begin{lemma}\label{lemma:shotton-smooth}
    If $(G, \sigma)$ is almost simply connected (Definition~\ref{def:almost-simply-connected}), $\mu$ is finite, and $z \in \mu(A)$, then $\sX_{G, \mu}^{\Sigma\textrm{-}\mathrm{reg}} \to B_{G, \mu}$ and $\sX_{G, \mu, z}^{\Sigma\textrm{-}\mathrm{reg}} \to B_{G, \mu, z}$, $(\Phi, \Sigma) \mapsto \Sigma$, are flat and surjective. If $Z(G)^\sigma/\mu$ is smooth, then these maps are smooth. In particular, $B_{G, \mu}$ and $B_{G, \mu, z}$ are $A$-flat, lci, and generically \'etale.
\end{lemma}

\begin{proof}
    For the first statement, one can apply the same argument as \cite[Theorem 3.7]{Shotton-irreducible} using Corollary~\ref{cor:asc-smooth} in place of \cite[Corollary 3.5]{Integral-Springer}, but we will offer a slightly different proof. Note first that we may reduce to the claim for $\sX_{G, \mu, z}$: indeed, the map $\sX_{G, \mu, z}^{\Sigma\textrm{-}\mathrm{reg}} \to B_{G, \mu, z}$ is obtained from $\sX_{G, \mu}^{\Sigma\textrm{-}\mathrm{reg}} \to B_{G, \mu}$ by base change along $B_{G, \mu, z} \to B_{G, \mu}$, so we may use Lemma~\ref{lemma:finite-gen-iso} and \cite[Seconde partie, Th\'eor\`eme (1.2.4)]{Raynaud-Gruson}.
    
    Note that the map $\chi_\sigma\colon G^{\sigma\textrm{-reg}} \to G/\!/_\sigma G$ is smooth by Corollary~\ref{cor:asc-adjoint-smooth}. In particular, it is enough to show that the projection $\mathrm{pr}_2\colon \sX_{G, \mu, z}^{\Sigma\textrm{-reg}} \to \chi_\sigma^{-1}(B_{G,\mu,z})$ is smooth and surjective. For surjectivity, we will show the stronger claim that $\mathrm{pr}_2$ is a surjection of fppf sheaves. Note that if $R$ is an $A$-algebra and $\Sigma \in G^{\sigma\textrm{-reg}}(R)$ is such that $\chi_\sigma\left(\prescript{\rm{Fr}^{-1}}{}{\left(z^{-1} \cdot \prod_{i=0}^{q-1}\prescript{\sigma^i}{}{\Sigma}\right)}\right) = \chi_\sigma(\Sigma)$, then it follows that $\prescript{\rm{Fr}^{-1}}{}{\left(z^{-1} \cdot \prod_{i=0}^{q-1}\prescript{\sigma^i}{}{\Sigma}\right)} \in G^{\sigma\textrm{-reg}}(R)$: indeed, it suffices to check this in the case that $R = k$ is an algebraically closed field. In this case, let $W_0$ be a finite quotient of $W_F$ through which the action factors, and let $(\Sigma, \sigma) = tu$ be the Jordan decomposition in $G(k) \rtimes W_0$. By assumption, $z \cdot \prescript{\rm{Fr}}{}{(tu)}$ and $(tu)^q$ have the same image in $(G \rtimes W_0)/\!/G = \bigsqcup_{w_0 \in W_0} G/\!/_{w_0} G$, so their semisimple parts must be $G$-conjugate, i.e., $z \cdot \prescript{\rm{Fr}}{}{t}$ and $t^q$ are $G(k)$-conjugate. Also, $Z_{Z_G(z\cdot\prescript{\rm{Fr}}{}{t})}(\prescript{\rm{Fr}}{}{u}) = \prescript{\rm{Fr}}{}{Z_G(tu)}$ and $Z_G((tu)^q)$ have the same dimension because $\chara k$ does not divide $q$: indeed, if $\chara k = 0$, then $u$ is $G(k)$-conjugate to $u^q$ by the existence of the exponential map, and if $\chara k \neq 0$ then $u$ is a power of $u^q$. Since the $G_k$-centralizer of $z \cdot \prescript{\rm{Fr}}{}{(tu)}$ is of minimal dimension in its component (since $tu$ is $\sigma$-regular) and it lies in the same component of $G_k$ as $(tu)^q$, the claim follows. It follows from Lemma~\ref{lemma:transp-fppf} that there exists an fppf extension $R'$ of $R$ and $\Phi \in G(R')$ such that $(\Phi, \Sigma) \in \sX_{G, \mu, z}^{\Sigma\textrm{-reg}}(R')$. This shows that $\mathrm{pr}_2$ is a surjection of fppf sheaves, and in particular it is surjective.

    Now let $R$ be the coordinate ring of $\chi_\sigma^{-1}(B_{G,\mu,z})$, let $\Sigma \in \chi_\sigma^{-1}(B_{G,\mu,z})(R)$ be the universal section, and let $R'$ be an fppf $R$-algebra such that $\mathrm{pr}_2^{-1}(\Sigma)(R')$ is non-empty, and let $(\Phi, \Sigma)$ lie in this set. It is clear from the functorial descriptions that if $R''$ is an $R'$-algebra, then the map $Z_{G, \sigma^q}(\prescript{\mathrm{Fr}}{}{\Sigma})(R'') \to \mathrm{pr}_2^{-1}(\Sigma)(R'')$, $g \mapsto (\Phi g, \Sigma)$ is a bijection, so $\mathrm{pr}_2^{-1}(\Sigma)_{R'} \cong Z_{G, \sigma^q}(\Sigma)_{R'}$. Now Theorem~\ref{theorem:regular-unipotent-centralizer} and Corollary~\ref{cor:asc-smooth} show that $Z_{G, \sigma^q}(\prescript{\mathrm{Fr}}{}{\Sigma})$ is a smooth $R$-group scheme, so $\mathrm{pr}_2$ is smooth, as desired.
    
    For the final claim, recall from above that the schemes $\sX_{G, \mu}$ and $\sX_{G, \mu, z}$ are $A$-flat and lci and generically smooth. In particular, $B_{G,\mu}$ and $B_{G,\mu,z}$ are $A$-flat, lci, and generically smooth by \cite[Theorem 23.7]{Matsumura}. It is easy to see that $B_{G,\mu}(\ov{K})$ and $B_{G,\mu,z}(\ov{K})$ are finite, so in fact $B_{G,\mu}$ and $B_{G,\mu,z}$ are generically \'etale.
\end{proof}

\begin{example}\label{example:universal-deformation-ring}
	Let $G = \PGL_2$ with trivial $W_F$-action, and let $\pi\colon \SL_2 \to G$ be the natural projection. Suppose $\ell = 2$. Note that Lemma~\ref{lemma:shotton-smooth} shows that if $K$ is ``large" then the irreducible components of $\sX_{\SL_2, \mu_2}$ (considered with reduced structure) have special fibers of multiplicity $1$, and in particular if $\sC_0$ is an irreducible component of $\sX_{\SL_2,\mu_2}$ and $x \in \sC_0(k)$ then $\widehat{\cO}_{\sC_0,x}$ has reduced special fiber. Let $C$ be an irreducible component of $(\sX_G)_K$ which is geometrically irreducible such that every point of $C(K)$ has second coordinate which is regular and not conjugate to $\begin{pmatrix} i&1\\&-i \end{pmatrix}$, where $i^2 = -1$. Suppose moreover that the closure $\ov{C}$ of $C$ in $\sX_G$ is such that there is an irreducible component $D \subset \ov{C}_k$ such that every point of $D(k)$ has second coordinate which is regular unipotent, and let $\sC$ be the open subscheme of $\ov{C} \cap \sX_G^{\Sigma\textrm{-reg}}$ obtained by removing every irreducible component other than $D$ from the special fiber. By Lemma~\ref{lemma:preimage-of-completion}, if $f\colon \sX_{\SL_2,\mu_2} \to \sX_G$ is the natural map and $x \in f^{-1}(\sC)(k)$ has image $y \in \sC(k)$, then the map $\Spec \widehat{\cO}_{f^{-1}(\sC),x} \to \Spec \widehat{\cO}_{\sC,y}$ is a $\mu_2$-torsor. By assumption on $C$, the preimage $f^{-1}(C)$ consists of two irreducible components in $(\sX_{\SL_2,\mu_2})_K$, while $f^{-1}(D)$ is irreducible. By the above discussion, we therefore have $\mu(\widehat{\cO}_{f^{-1}(\sC),x}/\pi) = 2$. Also, the map $(f^{-1}(\sC)_k)_{\rm{red}} \to (\sC_k)_{\rm{red}}$ is a $\mu_2$-torsor, so Lemma~\ref{lemma:preimage-of-completion} again implies that
	\[
	\deg(\Spec (\widehat{\cO}_{f^{-1}(\sC),x}/\pi)_{\rm{red}} \to \Spec (\widehat{\cO}_{\sC,y}/\pi)_{\rm{red}}) = 2.
	\]
	Note that $\Spec (\widehat{\cO}_{f^{-1}(\sC),x}/\pi)_{\rm{red}} \to \Spec (\widehat{\cO}_{\sC,y}/\pi)_{\rm{red}}$ is a homeomorphism, so $\Spec (\widehat{\cO}_{\sC,y}/\pi)_{\rm{red}}$ is integral and Lemma~\ref{lemma:multiplicity-behavior} shows that $\mu(\widehat{\cO}_{\sC,y}/\pi) = 2$. Since $\Spec \widehat{\cO}_{f^{-1}(\sC), x}$ has two irreducible components which are permuted by the $\mu_2$-action, it follows that $\Spec \widehat{\cO}_{\sC,y}$ is irreducible and thus generic local deformation rings are not formally smooth in this case.
\end{example}

One can perform a completely similar calculation in general when $G$ is of adjoint type and the universal cover is not \'etale, but it is not clear to the author even in the case $G = \SL_2$ ($\ell = 2$) how to describe the multiplicities in the complete local rings at $\Sigma$-(regular unipotent) points in a conceptual way. This is the reason we have not formulated a general statement concerning multiplicities for Galois deformation rings.

\bibliography{bibliography}

\end{document}